\documentclass[11pt,a4paper,reqno]{amsart} 
\usepackage{footnote}
\usepackage{amssymb}
\usepackage{latexsym}
\usepackage{amsmath}
\usepackage{mathrsfs}
\usepackage{mathabx}
\usepackage[X2,T1]{fontenc}
\usepackage[applemac]{inputenc}
\usepackage{cite}
\usepackage{yfonts}
\usepackage{amscd}
\usepackage{color}
\usepackage{euscript}

\usepackage{circuitikz}

\usepackage{calc}                   

\newcommand{\scal}[2]{\langle #1,#2\rangle}
\newcommand{\rr}[1]{\mathbf R^{#1}}
\newcommand{\zz}[1]{\mathbf Z^{#1}}

\newcommand{\nm}[2]{\Vert #1\Vert _{#2}}
\newcommand{\NM}[2]{\left \Vert #1\right \Vert _{#2}}

\newcommand{\op}{\operatorname{Op}}

\newcommand{\sets}[2]{\{ \, #1\, ;\, #2\, \} }

\newcommand{\ep}{\varepsilon}
\newcommand{\fy}{\varphi}
\newcommand{\cdo}{\, \cdot \, }
\newcommand{\supp}{\operatorname{supp}}

\newcommand{\wpr}{{\text{\footnotesize $\#$}}}

\newcommand{\ON}{\operatorname{ON}}

\newcommand{\vrum}{\vspace{0.1cm}}

\newcommand{\RE}{\operatorname{Re}}
\newcommand{\IM}{\operatorname{Im}}

\newcommand{\WL}{W\!\! L}

\newcommand{\GL}{\mathbf{M}}

\newcommand{\nn}[1]{{\mathbf N}^{#1}}

\newcommand{\maclB}{\mathcal B}

\newcommand{\maclH}{\mathcal H}

\newcommand{\maclK}{\mathcal K}

\newcommand{\maclM}{\mathcal M}

\newcommand{\maclS}{\mathcal S}
\newcommand{\maclT}{\mathcal T}
\newcommand{\maclV}{\mathcal V}

\newcommand{\mascA}{\mathscr A}
\newcommand{\mascB}{\mathscr B}

\newcommand{\mascF}{\mathscr F}
\newcommand{\mascI}{\mathscr I}
\newcommand{\mascM}{\mathscr M}
\newcommand{\mascP}{\mathscr P}
\newcommand{\mascS}{\mathscr S}

\newcommand{\fka}{\mathfrak a}
\newcommand{\fkb}{\mathfrak b}
\newcommand{\fkc}{\mathfrak c}

\numberwithin{equation}{section}          

\newtheorem{thm}{Theorem}
\numberwithin{thm}{section}

\newcommand{\rubrik}{}
\newtheorem{prop}[thm]{Proposition}
\newtheorem{cor}[thm]{Corollary}
\newtheorem{lemma}[thm]{Lemma}

\theoremstyle{definition}

\newtheorem{defn}[thm]{Definition}
\newtheorem{example}[thm]{Example}

\theoremstyle{remark}
\newtheorem{rem}[thm]{Remark}

\title{Factorizations for quasi-Banach
time-frequency spaces
and Schatten classes}

\author{Divyang G. Bhimani}

\address{Department of Mathematics,
Indian Institute of Science Education and Research, Pune, India}

\email{divyang.bhimani@iiserpune.ac.in}

\author{Joachim Toft}

\address{Department of Mathematics,
Linn{\ae}us University, V{\"a}xj{\"o}, Sweden}

\email{joachim.toft@lnu.se}




\begin{document}

\begin{abstract}
We deduce factorization
properties for Wiener amalgam spaces $\WL ^{p,q}$,
an extended family of modulation
spaces $M(\omega ,\mascB )$, and for Schatten symbols
$s_p^w$ in pseudo-differential
calculus under e.{\,}g. convolutions, twisted
convolutions and symbolic products. Here
$M(\omega ,\mascB )$ can be any quasi-Banach
Orlicz modulation space. For 
example we show that
$\WL ^{1,r}*\WL ^{p,q}=\WL ^{p,q}$ and
$\WL ^{1,r} \wpr s_p^w = s_p^w$ when 
$r\in (0,1]$,
$r\leq p,q< \infty$. In particular we improve
Rudin's identity $L^1*L^1=L^1$.
\end{abstract}

%
%

\maketitle

\par

\section{Introduction}\label{sec0}

\par

Several issues in science and technology 
concern topological algebras and modules
(see e.{\,}g. \cite{Bal,BecNarSuf,Fra}
and the references therein). In this context
a relevant and fascinating question
with an algebraic flavor arise quite 
naturally,
wether the algebra or module
under consideration possess the so-called factorization
property. That is whether they are factorization algebras
and factorization modules, or not. More specifically,
consider the pair $(\maclB ,\maclM ,\cdo )$ of
algebras and modules, where
$\maclB$ is a topological algebra,
$\maclM$ is a topological (right) module over $\maclB$
with multiplication $\cdot$ between elements in $\maclB$
and $\maclM$.
(See Sections \ref{sec1}, \ref{sec4} or
\cite{Hor1} for notations.)
Is it then true that for arbitrary elements $\phi \in \maclB$
and $f\in \maclM$, there are elements $\phi _j\in \maclB$
and $f_0\in \maclM$ such that
$$
\phi = \phi _1\phi _2
\quad \text{and}\quad
f=\phi _0\cdot f_0?
$$

\par

Evidently, if $\maclB$ is unital, i.{\,}e. it contains
the unitary element $\varrho$, then
$$
\phi = \varrho \phi
\quad \text{and}\quad
f=\varrho \cdot f,
$$
giving the factorization property. On the other hand,
if such unitary element is missing, the question on
factorization property becomes more subtle
and challenging, and less trivial.

\par

As examples, let $p\in [1,\infty )$, and
consider the pairs of Banach algebras and modules
\begin{equation}\label{Eq:ExAlgModPairs}
(L^\infty (\rr d),L^p (\rr d),\cdo ),
\quad
(L^1 (\rr d),L^p(\rr d),*),
\quad \text{and}\quad
(\mascI _2,\mascI _\infty ,\circ ).
\end{equation}
Here $\mascI _2$ and $\mascI _\infty$ are the
sets of Hilbert-Schmidt and continuous
operators on $L^2(\rr d)$, respectively,
and $\circ$ denotes the operator composition.
We also note that H{\"o}lder's and Young's inequalities
guarantee the Banach algebra and module
properties for the first two pairs in
\eqref{Eq:ExAlgModPairs}.

\par

Then the first pair in \eqref{Eq:ExAlgModPairs}
possess the factorization property because
the multiplicative identity $1$ is an element
in the Banach algebra $L^\infty (\rr d)$. 

\par

For the second pair in \eqref{Eq:ExAlgModPairs},
there is no identity element in the algebra $(L^1(\rr d),*)$.
However, despite this fact we indeed have
\begin{equation}\label{Eq:IntrConvIdent}
L^1(\rr d)*L^p(\rr d)= L^p(\rr d)
\end{equation}
when $1\le p<\infty$. Thereby it follows that this pair
possess the factorization property.

\medspace

\begin{circuitikz} \draw
(0,0.5) to[short, o-*] (1,0.5)
(0,1.5) to[short, o-*] (1,1.5)
(1,0) to (1,2)
to (3,2) to (3,0)
to (1,0)
(3,0.5) to[short, *-o] (4,0.5)
(3,1.5) to[short, *-o] (4,1.5)
(2,1.5) node{\text{Filter}}
(2,0.8) node{$\phi$}
(0,1.0) node{\text{$f_{\operatorname{in}}$ } } 
(5,1.0) node{\text{$f_{\operatorname{out}}=f_{\operatorname{in}}*\phi$ } } 
;
\end{circuitikz}

\noindent
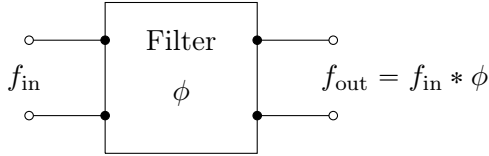
{{\small{\text{{\textbf{Picture 0.1}}
What kind of outgoing signals $f_{\operatorname{out}}$ can be}}}}
\\
{{\small{\text{constructed
from insignals $f_{\operatorname{out}}$ and filters with functions $\phi$?} } }}

\medspace

In signal analysis, a related question concerns
what kind of outgoing signals $f_{\operatorname{out}}$ belonging
to e.{\,}g. $L^1$, can be obtained from stationary filters. (See Picture 0.1.)
Suppose for example that we can choose filters with filter
functions $\phi$ and ingoing signal $f_{\operatorname{in}}$
being any $L^1$ functions. Then the answer on this, in
a pure theoretical framework, is that
any $L^1$ function can be obtained as
outgoing signal, in view of \eqref{Eq:IntrConvIdent}.
Evidently, in reality we may not choose filter functions and ingoing
signals as any $L^1$ functions.

\par

For the third pair in \eqref{Eq:ExAlgModPairs}
we have
$$
\mascI _2 \circ \mascI _\infty
=
\mascI _2
\subsetneq
\mascI _\infty ,
$$
which shows that it cannot possess the 
factorization
property. On the other hand, by the swap of the rules of
$\mascI _2$ and $\mascI _\infty$ in
\eqref{Eq:ExAlgModPairs} into the pair
$(\mascI _\infty ,\mascI _2,\circ)$, i.{\,}e. letting
$\mascI _2$ be a module over the ring $\mascI _\infty$,
we have
\begin{align}
\mascI _\infty \circ \mascI _2 &= \mascI _2.
\notag
\intertext{This is trivially true because because
$\mascI _\infty$ is unital. Less trivial is that indeed}
\mascI _\sharp \circ \mascI _2 &= \mascI _2,
\notag
\intertext{or more generally,}
\mascI _\sharp \circ \mascI _p &= \mascI _p,
\quad p\in [0,\infty ),
\label{Eq:IntroSchattenFact}
\end{align}
since $\mascI _\sharp$ is not unital. That is,
$(\mascI _\sharp ,\mascI _p,\circ)$ possess
factorization property. Here $\mascI _p$
is the space of Schatten-von
Neumann operators and $\mascI _\sharp$
is the space of compact operators on $L^2(\rr d)$.

\medspace

In the paper we deduce factorization properties for 
certain quasi-Banach spaces 
under various kinds of multiplications.

\par

In Section \ref{sec2} we consider such
factorizations for weighted versions
of Wiener-amalgam spaces
$\WL ^{p,q}(\rr d)$, $p,q\in (0,\infty ]$,
often denoted by
$W(L^p(\rr d),L^q(\rr d))$ in the literature,
with $L^p(\rr d)$ as local component and
$L^q(\rr d)$ as global component. That is,
$\WL ^{p,q}(\rr d)$ consists of all
Lebesgue measurable functions $f$
with finite quasi-norm
\begin{equation}\label{Eq:WienAmNormIntro}
\begin{aligned}
\nm f{\WL ^{p,q}}
&\equiv
\nm {a_f}{\ell ^q(\zz d)},
\\[1ex]
a_f (k) &= \nm f{L^p(k+Q)},\ k\in \zz d,\ Q=[0,1]^d 
\end{aligned}
\end{equation}
(see e.{\,}g. \cite{FeiGro1,FouSte,Gro2}).

\par

For $\WL ^{1,r}(\rr d)$, 
the usual mollifiers become bounded approximate identities
under convolutions (see Example \ref{Example: Mollifiers}). 
This leads to the following two theorems on factorizations
of Wiener-amalgam and classical modulation spaces.
The first result is essentially
Theorems \ref{Thm:SurjConvWien} in Section \ref{sec2}
after restricting ourselves to allow only trivial weights.

\par

\begin{thm}\label{Thm:SurjConvWienIntro}
Let $r\in (0,1]$, $p\in [1,\infty )$ and $q\in [r,\infty )$.
Then 
\begin{equation}\label{Eq:WienerAMIdentIntro}
\WL ^{1,r} * \WL ^{p,q}
=
\WL ^{p,q}.
\end{equation}
\end{thm}

\par

The next (second) result is essentially
Theorem \ref{Thm:ConvLebMod} in Section \ref{sec3},
after restricting ourselves to allow only trivial weights
and to classical modulation spaces.
Here we remark that Theorem \ref{Thm:ConvLebMod} deals
with factorizations for a large family of
quasi-Banach modulation spaces. (See
Definition \ref{Def:ModSpaces}.)

\par

\begin{thm}\label{Thm:ConvLebModIntro}
Suppose 
$r\in (0,1]$ and $p,q\in [r,\infty )$. 
Then
\begin{equation}
\WL ^{1,r} * M^{p,q}
=
M^{p,q}.
\end{equation}
\end{thm}

\par

Theorem \ref{Thm:SurjConvWienIntro} remains the same
after the convolution $*$ is replaced by the twisted
convolution $*_\sigma$. (See Theorem
\ref{Thm:SurjTwistConvWien}.)

\par

If $r=1$ and $q=p$ in Theorems \ref{Thm:SurjConvWienIntro}
and \ref{Thm:ConvLebModIntro}, we recover Rudn's identity
\begin{align}
L^1 * L^p &= L^p,
\label{Eq:RudinIdent}
\intertext{as well as reaching}
L^1 * M^{p,q} &= M^{p,q},
\label{Eq:ModLebIdentIntro}
\intertext{which can be obtained by
straight-forward applications of Hewitt's
factorization result \cite[(2.5) Theorem]{Hew}.
\newline
\indent
If instead $r\le 1$ but still supposing that
$q=p$ in Theorem \ref{Thm:SurjConvWienIntro},
then we obtain}
\WL ^{1,r} * L^p
&=
L^p,\qquad
r\le 1.
\label{Eq:RudinIdentImproved}
\end{align}
which is a strict improvement of 
\eqref{Eq:RudinIdent}, because
$\WL ^{1,r}\subsetneq L^1$ when $r<1$
(with continuous embeddings). In fact, if
$$
r<1
\quad \text{and}\quad
\{ t_j\} _{j=0}^\infty \in \ell ^1(\mathbf N)
\setminus \ell ^r(\mathbf N),
$$
then
$$
f=\sum _{j=0}^\infty t_j\chi _{j+[0,1]}
\in L^1(\mathbf R)\setminus
\WL ^{1,r}(\mathbf R),
$$
which shows that $\WL ^{1,r}\neq L^1$.
(Here $\chi _E$ is the characteristic function of the set $E$.)
This also shows that Theorem
\ref{Thm:ConvLebModIntro} is a strict
improvement of \eqref{Eq:ModLebIdentIntro}. We also 
observe that Theorems
\ref{Thm:SurjConvWienIntro} and
\ref{Thm:ConvLebModIntro},
and \eqref{Eq:RudinIdentImproved} are not 
reachable by the classical results in 
\cite{Coh,Hew,Rud1,Rud2}
when $r<1$. In fact, these earlier 
investigations do not allow quasi-Banach 
spaces like $\WL ^{1,r}$ which fails to be 
Banach spaces when $r<1$.

\par

We also observe that the assertion
in Theorems \ref{Thm:SurjConvWienIntro}
and \ref{Thm:ConvLebModIntro},
and the identities \eqref{Eq:RudinIdent}
and \eqref{Eq:RudinIdentImproved}
fail to hold true when $p=\infty$. (See 
Remark \ref{Rem:NoFactCond}.)

\par

In Section \ref{sec4} we consider factorizations
for Schatten-von Neumann symbol classes in
pseudo-differential calculus with respect to
convolutions, twisted convolutions and
symbolic products under compositions. For example, 
essential achievements here is given in
Theorem \ref{Thm:SchattenTwistEqu} which 
contains the following. Here $s_p^w$ is the
set of all symbols $\fka$ such that corresponding
Weyl operators $\op ^w(\fka )$ is a Schatten-von
Neumann operator of order $p\in (0,\infty ]$
on $L^2$, and $\fka \wpr \fkb$ is the Weyl
product of $\fka$ and $\fkb$ (see
\eqref{Eq:WeylNot}).

\par

\begin{thm}\label{Thm:SchattenTwistEquIntro}
Suppose $p\in (0,\sharp ]$ and $r\in (0,1]$.
Then 
\begin{align}
\WL ^{1,r} *_\sigma s_{p}^w 
&=
s_\sharp ^w*_\sigma s_p ^w
=
s_\infty ^w*_\sigma s_p^w
=
s_{p}^w
\label{Eq:TwistedSchattenIdentIntro}
\intertext{and}
\WL ^{1,r} \wpr s_{p}^w
&=
s_\sharp ^w\wpr  s_p^w
=
s_\infty ^w\wpr s_p^w
=
s_{p}^w.
\label{Eq:WeylProdSchattenIdentIntro}
\end{align}
\end{thm}

\par

Using techniques of symplectic Fourier 
transform, it follows
that the identities 
\eqref{Eq:TwistedSchattenIdentIntro}
and \eqref{Eq:WeylProdSchattenIdentIntro} in 
the previous
theorem are equivalent. Since
\begin{equation}
\label{Eq:InclWienLebCompSymb}
\WL ^{1,r}
\subsetneq
L^1
\subsetneq
s_\sharp ^w
\subsetneq
s_\infty ^w
\end{equation}
the identities in 
\eqref{Eq:TwistedSchattenIdentIntro}
and \eqref{Eq:WeylProdSchattenIdentIntro}
at a first glance might be surprising.
The last equality
$s_\infty ^w\wpr s_p^w = s_{p}^w$
in \eqref{Eq:WeylProdSchattenIdentIntro} is
straight-forward, using H{\"o}lder's 
inequality
$s_\infty ^w\wpr s_p^w \subseteq s_{p}^w$,
and that the identity
$\fka _0(x,\xi )\equiv 1$
for $\wpr$ belongs to $s_\infty ^w$.

\par

The identities $s_\sharp ^w\wpr s_p^w = s_{p}^w$
and $L^1\wpr s_{p}^w = s_{p}^w$
in \eqref{Eq:WeylProdSchattenIdentIntro},
in the case $p\ge 1$,
follow e.{\,}g. from \eqref{Eq:InclWienLebCompSymb},
the bounded approximate identity for $L^1$ and
Hewitt's result \cite[(2.5) Theorem]{Hew}.
Finally, the strongest identity 
$\WL ^{1,r} \wpr s_{p}^w=s_p^w$
in \eqref{Eq:WeylProdSchattenIdentIntro}
follow from 
the bounded approximate identity for $\WL ^{1,r}$
and Theorem \ref{Thm:ImproveHewitIntro}.

\par

An essential ingredient of our investigations is an extension
of Hewitt's factorization result in \cite[(2.5) Theorem]{Hew}
to the quasi-Banach case, given by the following.

\par

\begin{thm}\label{Thm:ImproveHewitIntro}
Suppose $\maclB$ is a quasi-Banach algebra
and $\maclM$ is a left quasi-Banach $\maclB$ module.
Also suppose $\maclB$ possess a bounded left
approximate identity
for $(\maclB ,\maclM)$. Then
$$
\maclB \cdot \maclM = \maclM .
$$
\end{thm}

\par

See Section \ref{sec1} for descriptions of bounded
left approximate identities.

\par

A deep generalization
of Theorem \ref{Thm:ImproveHewitIntro} was
deduced by Ansari-Piri in
\cite[Theorem 4.1]{Ans}, who proved that
 \cite[(2.5) Theorem]{Hew} and
Theorem \ref{Thm:ImproveHewitIntro}
can be extended to general \emph{F}-spaces. On the other
hand, \cite[(2.5) Theorem]{Hew} is more general compared
to the factorization
results by Salem in \cite{Sal}, Rudin in \cite{Rud1,Rud2} and Cohen
in \cite{Coh}. In particular, Rudin's identity
$$
L^1(G)*L^1(G)=L^1(G)
$$
when $G$ is any locally Euclidean Abelian group follows
from  \cite[(2.5) Theorem]{Hew}
(and thereby from Theorem  \ref{Thm:ImproveHewitIntro} as well).

\par

These factorization results have found 
immanence applications in science. 
To give a flavor we note that in 1964,
Varopoulos in \cite{Var} applies of Cohen's
factorization theorem to prove that every positive
linear functional on a Banach$^*$-algebra with a
bounded approximate identity is continuous. In
1976, Green in \cite{Gre} proves that every
maximal left ideal in a Banach algebra with a 
bounded
right approximate identity is closed.

\par

A more recent application concerns \cite{BhiHaq},
where such factorization properties are used to
construct negative results on well-posedness 
for Hartree equation in $L^p(\rr d)\cap L^2(\rr d)$
when $1<p\le 2$. See \cite[Remark 3.6]{BhiHaq}.

\par

For Segal algebras, Feichtinger, Graham
and Lakien showed in \cite{FeGrLa} that
factorizations can happen under rather
weak assumptions. We also remark that
there are examples on Banach
algebras having the factorization
property, but with absence of
bounded approximate units (see e.{\,}g.
\cite{Lei,Wil}).

\par

We have included a proof of
Theorem \ref{Thm:ImproveHewitIntro} in Appendix
\ref{App:A}, which follows the framework of Hewitt's proof
in \cite{Hew} in the case of Banach algebras.

\medspace

The paper is organized as follows. In Section \ref{sec1}
we introduce needed notations and recall basic facts for
quasi-Banach spaces, algebras and modules,
weight functions used in time-frequency analysis,
the Gelfand-Shilov space $\Sigma _1(\rr d)$
and its distribution space, and Wiener amalgam
and modulation spaces.

\par

In Section \ref{sec2} we apply
Theorem \ref{Thm:ImproveHewitIntro}
to prove factorization properties
for quasi-Banach Wiener amalgam under
convolutions and twisted convolution. (See Theorems
\ref{Thm:SurjConvWien} and \ref{Thm:SurjTwistConvWien}.)
In Section \ref{sec3} we deduce factorization properties
for a large class of quasi-Banach modulation spaces
under convolutions. (See Theorem \ref{Thm:ConvLebMod}.)
The class of modulation spaces is explained in
Subsection \ref{subsec1.1}, and is broader compared to
what is usually considered. In Appendix \ref{App:B} we
show some properties of such modulation spaces.

\par

In Section \ref{sec4} we first make a review of
Schatten-von Neumann symbol classes in calculi of
pseudo-differential operators and introduce related
items. Then we apply the results from
Section \ref{App:A} to prove factorization properties
for such symbol classes under convolutions,
twisted convolutions and symbol products
which are the operator compositions on the
symbol side. (See Theorems
\ref{Thm:FactSchattenWigner}--\ref{Thm:SchattenPseudoProdEqu}.)

%

\par

\section*{Acknowledgement}

\par

We are grateful to professors
H. G. Feichtinger and R. Fulsche for
several valuable  comments and advices
which lead to significantly improvements of 
the content and style.

\par

The first author is grateful to professor 
J. R. Patadia for introducing him 
factorization problems during his M.Phil. 
thesis. The first author is also grateful to professors
K. Ross, S. Thangavelu and P.K. Ratnakumar for several
discussions and their encouragement on the topic. 

\par

The second author was supported by 
Vetenskapsr{\aa}det (Swedish Science 
Council), within the project 2019-04890.  

\par

\section*{Conflict of interest}

On behalf of all authors, the corresponding author states that there is no conflict of interest.

\par

\section{Preliminaries}\label{sec1}

\par

In this section we first review some facts 
on quasi-Banach algebras and modules. Then 
we discuss weight functions. Thereafter we 
recall the definition and basic properties 
on a suitable Gelfand-Shilov space and
its distribution space to have in background.
Then we make a review of certain types of
Wiener amalgam spaces and consider a broad 
class of modulation spaces and discuss 
general properties of such spaces.

\par

\subsection{Quasi-Banach algebras and 
modules}\label{subsec1.1}

\par

Let $\maclB$ be a vector space over
$\mathbf C$.
(We only consider vector spaces over 
$\mathbf C$, but evidently
our analysis also holds for vector spaces 
over $\mathbf R$.)
A quasi-norm of order
$p\in (0,1]$, or a $p$-norm on $\maclB$, is 
a map
$f\mapsto \nm f{\maclB}$ from
$\maclB$ to $\mathbf R$ such that
\begin{alignat}{2}
\nm {\phi +\psi }{\maclB} &\le 2^{\frac 1p-1}
(\nm \phi {\maclB} + \nm \psi {\maclB}),&
\quad \phi ,\psi  &\in \maclB ,
\label{Eq:WeakTriangleIneq1}
\\[1ex] 
\nm {\alpha \phi}{\maclB} &= |\alpha |\, \nm \phi{\maclB}, &
\quad \phi &\in {\maclB},\ \alpha \in \mathbf C,
\label{Eq:NormMult}
\intertext{and}
\nm \phi {\maclB} &> 0, &
\quad \phi &\in {\maclB}\setminus \{ 0\} .
\label{Eq:NormNonDeg}
\end{alignat}
Evidently, $\nm \phi{\maclB}=0$, if and only if $\phi =0$, due to
\eqref{Eq:NormMult} and \eqref{Eq:NormNonDeg}.

\par

The vector space $\maclB$ with $p$-norm
$\nm  \cdo\maclB$ is called a quasi-Banach 
space of order $p$, if $\maclB$
is complete under the vector space topology 
defined by $\nm \cdo{\maclB}$. By
Aoki-Rolewicz theorem it follows
that for the quasi-Banach space $\maclB$ 
with quasi-norm
$\nm \cdo {\maclB}$ satisfying
\eqref{Eq:WeakTriangleIneq1}--
\eqref{Eq:NormNonDeg}, there is
an equivalent quasi-norm
of order $p$,
which additionally satisfies
\begin{equation}
\nm {\phi +\psi}{\maclB}^p \le \nm \phi{\maclB}^p + \nm \psi{\maclB}^p,
\quad \phi ,\psi \in \maclB .
\label{Eq:WeakTriangleIneq2}
\end{equation}
(See \cite{Aik,Rol}.)
From now on we assume that the quasi-norms
to our quasi-Banach spaces
are chosen such that 
\eqref{Eq:WeakTriangleIneq1}--\eqref{Eq:WeakTriangleIneq2} hold true.

\par

We have now the following definition.

\par

\begin{defn}\label{Def:QuasiBanAlgMod}
Let $p\in (0,1]$.
\begin{enumerate}
\item The algebra $\maclB$ with $p$-norm $\nm \cdo{\maclB}$
satisfying \eqref{Eq:WeakTriangleIneq1}--\eqref{Eq:WeakTriangleIneq2},
is called a \emph{quasi-Banach algebra} of order $p$, or
a $p$-Banach algebra, if $\maclB$
is a quasi-Banach space of order $p$ under $\nm \cdo{\maclB}$, and
there is a constant $C>0$ such that
\begin{equation}\label{Eq:QNormAlg}
\nm {\phi \psi }{\maclB} \le C\nm \phi{\maclB}\nm \psi{\maclB},
\quad \phi ,\psi \in \maclB \text ;
\end{equation}

\vrum

\item if $\maclB$ is a quasi-Banach algebra 
of order $p$,
then the left (right) $\maclB$ module 
$\maclM$ with $p$-norm
$\nm \cdo{\maclM}$ and multiplication
$\cdot$ is called a
\emph{left (right) quasi-Banach $\maclB$
module of order $p$}, or a $p$-Banach
module over $\maclB$,
if $\maclM$ is a quasi-Banach space under
$\nm \cdo{\maclM}$ of order $p$, and there 
is a constant $C>0$ such that
\begin{equation}\label{Eq:QNormMod}
\nm {\phi \cdot f}{\maclM} \le C\nm \phi{\maclB}\nm f{\maclM}
\quad
\big (
\nm {f\cdot \phi}{\maclM} \le C\nm \phi{\maclB}\nm f{\maclM},
\big ),
\quad \phi \in \maclB ,\ f\in \maclM .
\end{equation}
\end{enumerate}
\end{defn}

\par

If $\maclB$ and $\maclM$ in Definition
\ref{Def:QuasiBanAlgMod} are Banach spaces,
then $\maclB$ is called a Banach algebra and
$\maclM$ is called a Banach module over $\maclB$.

\par

\begin{rem}\label{Rem:QuasiBanAlgSubMod}
Let $\maclB$ and $\maclM$ be as in Definition
\ref{Def:QuasiBanAlgMod},
and let $f\in \maclM$. Then
\begin{equation}\label{Eq:BSpanElementM}
\maclM _f \equiv
\overline{
\sets {\phi \cdot f}{\phi \in \maclB}}
\subseteq \maclM
\end{equation}
is a quasi-Banach $\maclB$ submodule to 
$\maclM$ of order $p$. Here $\overline 
\Omega$ denotes the closure of $\Omega 
\subseteq \maclM$ in $\maclM$.
\end{rem}

\par

By replacing the quasi-norms $\nm \cdo{\maclB}$ and
$\nm \cdo{\maclM}$ in Definition \ref{Def:QuasiBanAlgMod}
with  $c\cdot \nm \cdo{\maclB}$ and
$c\cdot \nm \cdo{\maclM}$, respectively, for some $c>0$, it
follows that \eqref{Eq:QNormAlg} and \eqref{Eq:QNormMod}
hold true with $C=1$. That is, we may reduce ourselves to the
case when \eqref{Eq:QNormAlg} and \eqref{Eq:QNormMod}
take the forms
\begin{alignat}{2}
\nm {\phi \psi }{\maclB} &\le \nm \phi{\maclB}\nm \psi{\maclB}, &
\quad \phi ,\psi &\in \maclB
\tag*{(\ref{Eq:QNormAlg})$'$}
\intertext{and}
\nm {\phi \cdot f}{\maclM} &\le \nm \phi{\maclB}\nm f{\maclM}, &
\quad \phi &\in \maclB ,\ f\in \maclM .
\tag*{(\ref{Eq:QNormMod})$'$}
\end{alignat}

\par

We also have the following definition.

\par

\begin{defn}\label{Def:ApprLeftUnit}
Let $\maclB$ be a quasi-Banach algebra of order $p\in (0,1]$
and $\maclM$ be a left quasi-Banach $\maclB$ module of
order $p$.
\begin{enumerate}
\item $\maclB$ is said to
\emph{possess a bounded
left approximate identity or unit} of order 
$r>0$, if for every $\ep >0$
and finite set
$\{ \psi _1,\dots ,\psi _n\} \subseteq 
\maclB$, there is an
$\phi _\ep \in \maclB$ such that
\begin{equation}\label{Eq:DefApprLeftUnit}
\nm {\phi _\ep \psi _j-\psi _j}{\maclB}<\ep ,
\ j=1,\dots ,n,
\quad \text{and}\quad \nm {\phi _\ep}{\maclB}\le r.
\end{equation}

\vrum

\item $\maclB$ is said to
\emph{possess a bounded left approximate
identity or unit for $\maclM$} of order 
$r>0$, if for every $\ep >0$
and $f\in \maclM$, there is an
$\phi _\ep \in \maclB$ such that
\begin{equation}\label{Eq:DefApprLeftUnitMod}
\nm {\phi _\ep \cdot f-f}{\maclM}<\ep
\quad \text{and}\quad \nm {\phi _\ep}{\maclB}\le r.
\end{equation}

\vrum

\item $\maclB$ is said to
\emph{possess a bounded left approximate
identity or unit for $(\maclB ,\maclM )$} of 
order $r>0$, if
for every $\ep >0$, finite set
$\{ \psi _1,\dots ,\psi _n\} \subseteq 
\maclB$ and
$f\in \maclM$, there is an $\phi _\ep \in 
\maclB$ such that
\eqref{Eq:DefApprLeftUnit} and 
\eqref{Eq:DefApprLeftUnitMod}
hold true.
\end{enumerate}
\end{defn}

\par

\begin{rem}
We observe that a quasi-Banach algebra 
$\maclB$ has a bounded
left approximate identity, if and only if 
there exist a net
$\{\phi _{\lambda} \} _{\lambda \in \Lambda}
\subseteq \maclB$  and a constant $r>0$ such 
that
$$
\lim _{\lambda \in \Lambda} \phi _{\lambda} \psi = \psi,
\quad \psi \in \maclB
$$
and $\|\phi _{\lambda}\|_{\maclB} \leq r$ for every
$\lambda \in \Lambda$. (See \cite{DorWic} for details.)
\end{rem}

\par

We have now the following factorization result, which is an
extension of \cite[(2.5) Theorem]{Hew}, but a special case of
\cite[Theorem 4.1]{Ans}.

\par

\begin{thm}\label{Thm:ApprIdent}
Suppose that $\maclB$ is a quasi-Banach algebra of order $p\in (0,1]$
and that $\maclM$ is a left quasi-Banach $\maclB$ module of order
$p$. Also suppose that $\maclB$ possess
a bounded left approximate identity
for $(\maclB ,\maclM)$ of order $r>0$.
Let $f\in \maclM$, $\maclM _f$
be the same as in \eqref{Eq:BSpanElementM}, 
set
$$
r_0=\frac r{((2r+1)^p-(2r)^p)^{\frac 1p}},
$$
and let $\ep >0$ be arbitrary. Then there exist elements
$\psi \in B_{r_0,\maclB}$ and $g\in \maclM _f$ such that
the following is true:
\begin{enumerate}
\item $f=\psi \cdot g$;

\vrum

\item $\nm {f-g}{\maclM}<\ep$.
\end{enumerate}
\end{thm}

\par

For conveniency to the reader, we give a proof of
Theorem \ref{Thm:ImproveHewitIntro} in Appendix
\ref{App:A}, following the framework of Hewitt's proof
in \cite{Hew} in the restricted case of Banach algebras.

\par

By Theorem \ref{Thm:ApprIdent} it follows 
that a quasi-Banach algebra having either
a bounded left or right approximate
identity is a factorization algebra
(cf. Definition \ref{Def:FactorAlgebras}). 
However,  the converse may not be true,
i.{\,}e., a quasi-Banach algebra which is a 
factorization algebra, might not contain
any bounded left or right approximate 
identity (see e.{\,}g. \cite{Lei,Wil}).
On the other hand, under
suitable factorization assumptions,
the converse is true, as in the following 
proposition.

\par

\begin{prop}\label{Prop:BoundedApprNeed}
Let $\maclB$ be a quasi-Banach algebra and 
suppose there is a constant $C>0$ such that
for every $\phi \in \maclB$
and $\delta>0$, there are elements $\phi _0,  
\psi \in \maclB$
such that 
$$
\phi = \phi_0 \psi ,  \quad \nm \psi {\maclB} 
\leq C,
\quad \text{and} \quad  \nm {f-g}{\maclB} < 
\delta.
$$
Then $\maclB$ has a bounded left approximate identity.
\end{prop}

\par

A proof of Proposition
\ref{Prop:BoundedApprNeed}, in the
restricted case when $\maclB$ is
a Banach algebra, can be found in
e.{\,}g. \cite{DorWic}. For general $\maclB$,
the result follows by similar arguments
as in \cite{DorWic}. We omit the details.

\par

\begin{example}\label{Example: Mollifiers}
If $p\in [1,\infty )$, then $\maclB=L^1(\rr d)$ is a Banach algebra
under convolution, and $\maclM = L^p(\rr d)$ is a Banach module
to $L^1(\rr d)$ under convolution. 
A classical example on bounded approximate identity
for such algebras and modules are
mollifiers
\begin{equation}\label{Eq:ExampBoundIdent}
\{ \phi _\ep \} _{\ep >0},
\quad \text{where}\quad
\phi _\ep (x) =\ep ^{-d}\phi ({\textstyle{\frac x\ep}}),\ 
\phi \in L^1(\rr d),\ \int _{\rr d}\phi (x)\, dx=1.
\end{equation}

\par

More generally, let $r\in (0,1]$ and $p\in [1,\infty )$, $q\in [r,\infty )$,
and $\maclB = \WL ^{1,r}(\rr d)$
and $\maclM = \WL ^{p,q}(\rr d)$ be as in
Theorem \ref{Thm:SurjConvWienIntro} in the introduction.
Then $\maclB$ is a quasi-Banach algebra and
$\maclM$ is a quasi-Banach module over $\maclB$ under
convolution. (See Proposition \ref{Prop:ConvWien} in Section
\ref{sec2}.) Again it follows that the mollifiers in
\eqref{Eq:ExampBoundIdent} are bounded left identities,
provided, more restrictive, $\phi \in \WL ^{1,r}(\rr d)$.
Here we also observe that $\maclB$ and $\maclM$
are \emph{not} Banach spaces when $r<1$ and
$\min (p,q)<1$.
\end{example}


\par

We also need to extend our quasi-Banach
algebras where a unit element is included.

\par

\begin{defn}\label{Def:ExtABAlgebraUnit}
Let $\maclB$ be a quasi-Banach algebra of order $p\in (0,1]$,
and let $\maclM$ be a left quasi-Banach $\maclB$ module of
order $p$.
\begin{enumerate}
\item if $\maclB$ is unital, then let $\maclB _E=\maclB$,
and let $\varrho$ denote the
unit element in $\maclB _E=\maclB$.

\vrum

\item  if $\maclB$ is not unital and $\varrho \notin \maclB$,
then let $\maclB _E
\equiv (\mathbf C\varrho ) \oplus \maclB$
be the algebra, equipped with the
operations
$$
(s\varrho +\phi )(t\varrho + \psi )
= st\varrho + (\phi \psi +t\phi +s\psi ),
\qquad s,t\in \mathbf C,\
\phi ,\psi \in \maclB ,
$$
and with quasi-norm
\begin{equation}
\label{Eq:ExtABAlgebraUnitNorm}
\nm {t\varrho +\phi}{\maclB _E}
\equiv
\left (
|t|^p +\nm \phi{\maclB}^p
\right )^{\frac 1p},
\qquad
\phi \in \maclB ,\ t\in \mathbf C .
\end{equation}

\vrum

\item $\maclB _E$ acts on $\maclM$
by the formula
$$
(t\varrho + \phi )\cdot f
=
tf+\phi \cdot f,
\qquad t\in \mathbf C ,\ 
\phi \in \maclB ,\ f\in \maclM .
$$
\end{enumerate}
\end{defn}

\par

It is evident that $\maclB _E$ in 
Definition \ref{Def:ExtABAlgebraUnit}
is a unital quasi-Banach
algebra of order $p$ containing $\maclB$, 
and with identity element $\varrho$.
For the quasi-norm
in \eqref{Eq:ExtABAlgebraUnitNorm} for $\maclB _E$
restricted to $\maclB$ one has
$$
\nm \phi{\maclB _E}=\nm \phi{\maclB}
\quad \text{when}\quad
\phi \in \maclB .
$$

\par

\begin{example}\label{Ex:BanAlgUnital}
The space $\maclB =
L^1(\rr d)$ is a Banach algebra
under convolution. Since $L^1(\rr d)$
does not contain any identity element,
it follows that
$$
\maclB _E=(\mathbf C\varrho )
\oplus \maclB 
$$
where $\varrho = \delta _0$ is the
Dirac's delta function.
In particular, $\maclB _E\neq \maclB $
in this case.

\par

Let instead $\maclB$ be the
Banach space $\maclM (\rr d)$ of
(complex-valued) measures on $\rr d$
with bounded mass. Recall that
the convolution
$\mu *\nu \in \maclM (\rr d)$ of
$\mu ,\nu \in \maclM (\rr d)$ is defined
by
$$
\scal {\mu *\nu}\fy
=
\iint _{\rr {2d}}\fy (x+y)
\, d\mu (x)d\nu (y),
\quad \fy \in C_0(\rr d).
$$
Then
\begin{itemize}
\item $(\maclM (\rr d),*)$ is a commutative
Banach superalgebra of the Banach
algebra $(L^1(\rr d),*)$, in view of
Fubini's theorem. Here we identify
any $f\in L^1(\rr d)$ with the measure
$\mu _f(x)=f(x)\, dx \in \maclM (\rr d)$, and observe
that $\nm {\mu _f}{{\maclM}}=\nm f{L^1}$.

\vrum

\item $\maclB _E=\maclB$ in this case, because
$\varrho = \delta _0\in \maclM (\rr d)$.
\end{itemize}
\end{example}

\par

Our investigations concerns factorization
properties for certain algebras in the 
following sense.

\par

\begin{defn}\label{Def:FactorAlgebras}
Let $\mascA$ be an algebra and
$\mascM$ be a left (right) module over
$\mascA$.
\begin{enumerate}
\item $\mascA$ is called a
\emph{factorization algebra} or
\emph{possess factorization property}
if for every $\phi \in \mascA$,
there are $\phi _1,\phi _2 \in \mascA$ such 
that
$\phi =\phi _1\phi _2$.

\vrum

\item $\mascM$ is called a
\emph{left (right) factorization module} or
\emph{possess factorization property}
if for every $f \in \mascM$,
there are $\phi _0\in \mascA$ 
and $f_0\in \mascM$ such that
$f =\phi _0\cdot f_0$
($f =f_0\cdot \phi _0$).
\end{enumerate}
\end{defn}

\par

In some examples later on, it is convenient
to have the following abstract result in hands.

\par

\begin{prop}\label{Prop:AlgCont}
Let $\mascA$ be an algebra, $\mascA _0$
a factorization subalgebra to $\mascA$,
$G$ be an ordered semi-group with ordering
$\prec$ and
let $\{ \mascA _g\} _{g\in G}$ be a family
of subalgebras to $\mascA$ such that
the following is true:
\begin{itemize}
\item if $g_1\prec g_2$, then $\mascA _{g_2}
\subseteq \mascA _{g_1}$;

\vrum

\item if $g\in G$ and $m,n\in \mathbf N$
satisfy $m<n$, then $g^m\prec g^n$;

\vrum

\item if $g_0\in G$ is fixed, then
\begin{equation}\label{Eq:IntersecIdent}
\bigcap _{n\in \mathbf Z_+} \mascA _{g_0^n}
=
\bigcap _{g\in G} \mascA _g.
\end{equation}
\end{itemize}
Then the following conditions are
equivalent:
\begin{enumerate}
\item $\mascA _0\subseteq \mascA _{g_0}$
for some $g_0\in G$;

\vrum

\item $\mascA _0\subseteq \mascA _g$ for every
$g\in G$.
\end{enumerate}
\end{prop}

\par

\begin{proof}
Evidently, (2) implies (1). Therefore
assume that (1) holds, and let
$\phi \in \mascA$ be arbitrary.
Also let $n\in \mathbf Z_+$ be arbitrary.
Since $\mascA$ is a factorization algebra, 
there are $\phi _1,\dots ,\phi _n
\in
\mascA _0\subseteq \mascA _{g_0}$ such that
$$
\phi = \phi _1\cdots \phi _n
\in \mascA _{g_0}\cdots \mascA _{g_0}
= \mascA _{g_0^n}.
$$
Hence $\phi \in \mascA _{g_0^n}$ for every
$n\in \mathbf Z_+$. By 
\eqref{Eq:IntersecIdent}
it follows that $\phi \in \mascA _g$ for 
every $g\in G$. Since $\phi \in \mascA _0$
was arbitrarily chosen it follows that
$\mascA _0\subseteq \mascA _g$ for every
$g\in G$, and the result follows.
\end{proof}

\par

\subsection{Weight functions}\label{subsec1.2}

\par

A \emph{weight} or \emph{weight function} $\omega$ on $\rr d$ is a
positive function such that $\omega ,1/\omega \in
L^\infty _{loc}(\rr d)$. Let $\omega$ and $v$ be weights on $\rr d$.
Then $\omega$ is called \emph{$v$-moderate} or \emph{moderate},
if
\begin{equation}\label{Eq:Moderate}
\omega (x+y)\lesssim \omega (x) v(y),\quad x,y\in \rr d .
\end{equation}
Here $f(\theta )\lesssim g(\theta )$ means
that $f(\theta )\le cg(\theta)$ for some
constant $c>0$ which is independent of $\theta$
in the domain of $f$ and $g$. We also write
$f(\theta )\asymp g(\theta )$ when $f(\theta )\lesssim
g(\theta )\lesssim f(\theta )$.

\par

The function $v$ is called \emph{submultiplicative}, if
it is even and \eqref{Eq:Moderate} holds for $\omega =v$. We notice that
\eqref{Eq:Moderate} implies that if $v$ is submultiplicative on $\rr d$, then
there is a constant $c>0$ such that
$v(x)\ge c$ when $x\in \rr d$.

\par

We let $\mascP _E(\rr d)$ be the set of all moderate weights on
$\rr d$. We recall that if $\omega \in \mascP _E(\rr d)$, then
$\omega$ is moderated by $v(x)=e^{r|x|}$, provided the constant
$r\ge 0$ is chosen large enough. A combination of this fact
with \eqref{Eq:Moderate} implies that
\begin{equation}\label{Eq:ModWeightLim}
e^{-r|x|}\lesssim \omega (x)\lesssim e^{r|x|},
\quad \omega \in \mascP _E(\rr d),
\end{equation}
where $r$ depends on $\omega$. (See \cite{Gro2}.)

\par

\subsection{A Gelfand-Shilov space}\label{subsec1.3}

\par

In several parts of our investigations it is suitable to
perform the discussions in the framework of
the Gelfand-Shilov space $\Sigma _1(\rr d)$
of Beurling type, which consists of all $f\in C^\infty (\rr d)$
such that
\begin{equation}\label{Eq:NormSpecialGS}
\sup _{\alpha ,\beta \in \nn d}\left (
\frac {\nm {x^\alpha D^\beta f}
{L^\infty}}{h^{|\alpha +\beta |}\alpha !\beta !}
\right )
\end{equation}
is finite for every $h>0$ (cf. 
\cite{Pil1}).
It follows that $\Sigma _1(\rr d)$
is a Fr{\'e}chet space under the
semi-norms in
\eqref{Eq:NormSpecialGS}. The (strong)
dual of $\Sigma _1(\rr d)$, i.{\,}e.
the (Gelfand-Shilov) distribution
space of $\Sigma _1(\rr d)$,
is denoted by $\Sigma _1'(\rr d)$.
It follows that the duality between 
$\Sigma _1(\rr d)$ and
$\Sigma _1'(\rr d)$ is given by a unique 
extension
of the $L^2$ scalar product
$(\cdo ,\cdo )_{L^2}$ on
$\Sigma _1(\rr d)\times
\Sigma _1(\rr d)$. We have
$$
\Sigma _1(\rr d)\subsetneq \mascS (\rr d)
\subsetneq
L^2(\rr d)
\subsetneq
\mascS '(\rr d)
\subsetneq
\Sigma _1'(\rr d)
$$
with continuous and dense embeddings.

\par

The spaces $\Sigma _1(\rr d)$ and $\Sigma _1'(\rr d)$
possess several convenient properties and
characterizations. If $\mascF$ denotes the
Fourier transform
$$
(\mascF f)(\xi ) = \widehat f(\xi )
\equiv
(2\pi )^{-\frac d2}\int _{\rr d}f(x)e^{-i\scal x\xi}\, dx 
$$
when $f\in L^1(\rr d)$, then $\mascF$
restricts to homeomorphisms
on $\mascS (\rr d)$ and on $\Sigma _1(\rr d)$.
The definition of $\mascF$ on $\Sigma _1(\rr d)$
extends uniquely to homeomorphisms
on $\mascS '(\rr d)$ and on $\Sigma _1'(\rr d)$,
and to an isometric bijection on $L^2(\rr d)$.
The inverse of $\mascF$ is given by
$$
\mascF ^{-1} = U\circ \mascF = \mascF \circ U,
\quad \text{when}\quad
(Uf)(x)=f(-x),\ f\in \Sigma _1'(\rr d).
$$
In particular, the Fourier's inversion formula is given by
$$
f(x) = (\mascF ^{-1}\widehat f)(x)
=
(2\pi )^{-\frac d2}\int _{\rr d}\widehat f(\xi )e^{i\scal x\xi}\, d\xi ,
$$
when $\widehat f\in L^1(\rr d)$.

\par

The following result shows that $\Sigma _1(\rr d)$
and $\Sigma _1'(\rr d)$ may in convenient ways be characterized
by growth and decay conditions on their Fourier transforms
and their short-time Fourier transform. Here recall that
the short-time Fourier transform of a
distribution $f\in \Sigma _1'(\rr d)$ with respect
to the window function
$\phi \in \Sigma _1(\rr d)\setminus 0$ is defined as
\begin{alignat*}{2}
V_\phi f(x,\xi )
&\equiv
\mascF (f\cdot \overline{\phi (\cdo -x)})(\xi ), &
\quad
x,\xi &\in \rr d.
\intertext{We note that for $f\in L^p(\rr d)$ for some
$p\in [1,\infty ]$, then}
V_\phi f(x,\xi )
&\equiv
(2\pi )^{-\frac d2}\int _{\rr d}
f(y)\overline{\phi (y-x)}e^{-i\scal y\xi}\, d\xi , &
\quad
x,\xi &\in \rr d.
\end{alignat*}
In some parts of our exposition it
is more convenient
to use the transform
$$
\maclT _\phi f(x,\xi )
\equiv
\mascF (f(\cdo +x)\cdot \overline{\phi})(\xi )
= e^{i\scal x\xi}V_\phi f(x,\xi ),
\quad
x,\xi \in \rr d,
$$
instead of the short-time Fourier
transform, because of the simple
convolution relation
\begin{equation}\label{Eq:TOpConv}
(\maclT _\phi (f*g))(x,\xi )
=
\big ( (\maclT _\phi f)(\cdo ,\xi )*g \big )(x).
\end{equation}

\par

\begin{prop}\label{Prop:GSSpacesChar}
Let $\phi \in \Sigma _1(\rr d)\setminus 0$ be
fixed and let $f\in \Sigma _1'(\rr d)$. Then the following
is true:
\begin{enumerate}
\item it holds
$$
|V_\phi f(x,\xi )|
=
|\maclT _\phi f(x,\xi )|
\lesssim e^{r(|x|+|\xi |)}
$$
for some $r>0$;

\vrum

\item $f\in \Sigma _1(\rr d)$, if and only if
$$
|V_\phi f(x,\xi )|
=
|\maclT _\phi f(x,\xi )|
\lesssim e^{-r(|x|+|\xi |)}
$$
for every $r>0$;

\vrum

\item $f\in \Sigma _1(\rr d)$, if and only if
$$
|f(x)| \lesssim e^{-r|x|}
\quad \text{and}\quad
|\widehat f(\xi )| \lesssim e^{-r|\xi |}
$$
for every $r>0$.
\end{enumerate}
\end{prop}

\par

We refer to \cite{Toft10} for the proof of (1),
\cite{GroZim} for the proof of (2) and
\cite{ChuChuKim} for the proof of (3) 
in Proposition \ref{Prop:GSSpacesChar}.

\par

\subsection{Wiener amalgam and modulation spaces}\label{subsec1.4}

\par

Let $p\in [1,\infty ]$, $q\in (0,\infty ]$ and
$\omega \in \mascP _E(\rr d)$.
Then the Wiener amalgam space
$\WL ^{p,q}_{(\omega )}(\rr d)$
consists of all measurable functions $f$ on
$\rr d$ such that
$\nm f{\WL ^{p,q}_{(\omega )}}$ is 
finite, where
\begin{equation}\label{Eq:WienAmNorm}
\begin{aligned}
\nm f{\WL ^{p,q}_{(\omega )}}
&\equiv
\nm {a_f}{\ell ^q_{(\omega )}(\zz d)},
\\[1ex]
a_f (k) &= \nm f{L^p(k+Q)},\quad k\in \zz d,\ Q=[0,1]^d .
\end{aligned}
\end{equation}
The set $\WL ^{p,q}_{(\omega )}(\rr d)$
with quasi-norm \eqref{Eq:WienAmNorm}
is a quasi-Banach
space of order $\min (q,1)$. In particular,
$\WL ^{p,q}_{(\omega )}(\rr d)$
is a Banach space when $q\ge 1$.
We observe that if $p,q<\infty$, then
$C_0^\infty (\rr d)$ and $\Sigma _1(\rr d)$
are dense in $\WL ^{p,q}_{(\omega )}(\rr d)$. 
For convenience we set
$\WL ^{p,q}=\WL ^{p,q}_{(\omega )}$
when $\omega (x)=1$ for every $x\in \rr d$.
We also recall that
$\WL ^{p,q}_{(\omega )}(\rr d)$ is often
denoted by
\begin{equation*} \label{Eq:WienAmalgDefLit}
W(L^p(\rr d),L^q_{(\omega )}(\rr d)) =
W(L^p(\rr d),\ell ^q_{(\omega )}(\zz d)),
\end{equation*}
in the literature. (See e.{\,}g.
\cite{Fei3,FeiGro1}.) We observe that
$\WL ^{p,p}_{(\omega )}(\rr d)$
is equal to $L^p_{(\omega )}(\rr d)$,
the set of all measurable functions
$f$ on $\rr d$ such that
$\nm f{L^p_{(\omega )}}
\equiv \nm {f\cdot \omega}{L^p}$
is finite. Hence the
Wiener amalgam spaces extend the notion on
Lebesgue spaces.

\par

Next we recall the definition of 
\emph{classical modulation
spaces} (see e.{\,}g. \cite{Fei1,GaSa}).
Let $p,q\in (0,\infty ]$
be fixed, and set
$$
\nm F{L^{p,q}(\rr {2d})}
\equiv
\nm {f_p}{L^q(\rr d)},
\qquad
f_p(\xi )
\equiv
\nm {F(\cdo ,\xi )}{L^p(\rr d)},
$$
when $F$ is measurable on $\rr {2d}$. Then
$L^{p,q}(\rr {2d})$ is the quasi-Banach space
with quasi-norm $\nm \cdo{L^{p,q}}$ which
consists of all measurable functions $F$
on $\rr {2d}$ such that $\nm F{L^{p,q}(\rr {2d})}$
is finite. We observe that if $p,q\ge 1$, then
$L^{p,q}(\rr {2d})$ is a Banach space and
$\nm \cdo{L^{p,q}}$ is a norm.

\par

\begin{defn}\label{Def:ModSpacesClassical}
Let $p,q\in (0,\infty ]$,
$\phi \in \Sigma _1(\rr d)\setminus 0$
and let $\omega \in \mascP _E(\rr {2d})$.
\begin{enumerate}
\item
The modulation space $M^{p,q}_{(\omega )}(\rr d)$
consists of all $f\in \Sigma _1'(\rr d)$
such that
\begin{equation}\label{Eq:ClassicModNorm1}
\nm f{M^{p,q}_{(\omega )}}
\equiv
\nm {\maclT _\phi f\cdot \omega}{L^{p,q}}
\asymp
\nm {V_\phi f\cdot \omega}{L^{p,q}}
\end{equation}
is finite.

\vrum

\item The modulation space $W^{p,q}_{(\omega )}(\rr d)$
consists of all $f\in \Sigma _1'(\rr d)$
such that
\begin{equation}\label{Eq:ClassicModNorm2}
\nm f{W^{p,q}_{(\omega )}}
\equiv
\nm {G}{L^{q,p}}<\infty ,
\quad \text{where}\quad
G(\xi ,x) = V_\phi f(x,\xi )\cdot \omega (x,\xi ).
\end{equation}
\end{enumerate}
\end{defn}

\par

For convenience we set
$M^p_{(\omega )}=M^{p,p}_{(\omega )}=W^{p,p}_{(\omega )}$.
We also set
$$
M^{p,q}_{(\omega )}=M^{p,q},
\quad
W^{p,q}_{(\omega )}=W^{p,q}
\quad \text{and}\quad
M^p_{(\omega )}=M^p,
\quad
\text{when}\quad \omega =1.
$$

\par

We observe that the modulation spaces
in Definition \ref{Def:ModSpacesClassical}
(2) is certain types of Wiener amalgam
spaces (see e.{\,}g. \cite{Fei1}).
The
modulation spaces in Definition
\ref{Def:ModSpacesClassical}
(1) were introduced in \cite{Fei3}
by Feichtinger, and is often refer as
\emph{classical modulation spaces}.

\par

There are several extensions of classical
modulation spaces (see e.{\,}g.
\cite{BasCor1,Fei4,FeiGro1,PfeTof,Rau2,Toft10,TofUst,ToUsNaOz}).
We shall consider a broader family of
modulation spaces,
parameterized by quasi-Banach function
spaces given in the following definition.

\par

\begin{defn}\label{Def:BFSpaces1}
Let $\mascB$ be a quasi-Banach
space of measurable
functions on $\rr d$ of order $r\in (0,1]$ containing
$\Sigma _1(\rr d)$ with continuous embedding,
and let $v _0\in\mascP _E(\rr d)$.
Then $\mascB$ is called a \emph{translation invariant
quasi-Banach function space on $\rr d$},
or \emph{invariant QBF space on $\rr d$},
(with respect to $v_0$ of order $r$) if there is a constant
$C$ such that the following is true:
\begin{enumerate}
%
%
\item if $x\in \rr d$ and $f\in \mascB$,
then $f(\cdo -x)\in \mascB$, and 
\begin{equation}\label{Eq:TranslInv}
\nm {f(\cdo -x)}{\mascB}\le Cv_0(x)\nm {f}{\mascB}\text ;
\end{equation}

\vrum

\item if $f$ and $g$ are measurable on $\rr d$ which
satisfy $g\in \mascB$ and $|f| \le |g|$, then
$f\in \mascB$ and
$$
\nm f{\mascB}\le C\nm g{\mascB}\text .
$$
\end{enumerate}

\par

An invariant QBF space on $\rr d$ which is a Banach space is called
a \emph{translation invariant Banach function space on $\rr d$}
or \emph{invariant BF space on $\rr d$}.
\end{defn}

\par

\begin{defn}\label{Def:ModSpaces}
Let $\mascB$ be an invariant
QBF space on $\rr {2d}$ of order $r\in (0,1]$
and let $\phi \in \Sigma _1(\rr d)\setminus 0$.
\begin{enumerate}
\item
The modulation space $M (\omega ,\mascB)$
consists of all $f\in \Sigma _1'(\rr d)$
such that
\begin{equation}\label{Eq:ModNorm}
\nm f{M(\omega ,\mascB)}
\equiv
\nm {\maclT _\phi f\cdot \omega}{\mascB}
\asymp
\nm {V_\phi f\cdot \omega}{\mascB}
\end{equation}
is finite. 

\vrum

\item The modulation space $M_0(\omega ,\mascB)$
is the completion of $\Sigma _1(\rr d)$ under the norm
$\nm \cdo{M(\omega ,\mascB)}$.

\vrum

\item The modulation spaces $M(\omega ,\mascB )$
and $M_0(\omega ,\mascB )$ are called \emph{normal} if
$M(\omega ,\mascB ) \hookrightarrow \Sigma _1'(\rr d)$
and $M_0(\omega ,\mascB ) \hookrightarrow \Sigma _1'(\rr d)$,
respectively.
\end{enumerate}
\end{defn}

\par

\begin{rem}
Let $p\in (0,\infty ]$ be fixed, $v_r(x,\xi )= e^{r(|x|+|\xi |)}$
when $r>0$ is fixed and $x,\xi \in \rr d$. Then $\Sigma _1'(\rr d)$ is the
inductive limit of $M^p _{(1/v_r)}(\rr d)$ with respect to
$r>0$ (see e.{\,}g. \cite{Toft10}). This implies that $M(\omega ,\mascB)$
is normal, if and only if
\begin{equation}\label{Eq:NormalEquiv1}
M(\omega ,\mascB )
\hookrightarrow
M^p _{(1/v_r)}(\rr d),
\end{equation}
for some $r>0$, and then
\begin{equation}\label{Eq:NormalEquiv2}
\nm f{M^p _{(1/v_r)}}
\lesssim
\nm f{M(\omega ,\mascB )},
\qquad f\in \Sigma _1'(\rr d).
\end{equation}

\par

We also observe that if $\mascB$ is as in
Definition \ref{Def:ModSpaces}, then
\begin{equation}\label{Eq:GSinMod}
\Sigma _1(\rr d)\hookrightarrow M(\omega ,\mascB).
\end{equation}
This follows by a straight-forward
combination of \eqref{Eq:ModWeightLim},
Proposition \ref{Prop:GSSpacesChar}
and \eqref{Eq:TranslInv}. Since $\Sigma _1(\rr d)$
is the projective limit of $M^p _{(v_r)}(\rr d)$
with respect to $r>0$, it follows that 
$M(\omega ,\mascB)$
is normal, if and only if
\begin{equation}\tag*{(\ref{Eq:NormalEquiv1})$'$}
M^p _{(v_r)}(\rr d)
\hookrightarrow
M(\omega ,\mascB )
\hookrightarrow
M^p _{(1/v_r)}(\rr d),
\end{equation}
for some $r>0$, and then
\begin{equation}\tag*{(\ref{Eq:NormalEquiv2})$'$}
\nm f{M^p _{(1/v_r)}}
\lesssim
\nm f{M(\omega ,\mascB )}
\lesssim
\nm f{M^p _{(v_r)}},
\qquad f\in \Sigma _1'(\rr d).
\end{equation}
\end{rem}

\par

\begin{rem}
Let $\omega \in \mascP _E(\rr {2d})$ and $\mascB$ be an
invariant QBF spaces on $\rr {2d}$.
By the previous remark it follows that a sufficient condition
for $M(\omega ,\mascB )$ to be normal is that
$$
\mascB \hookrightarrow L^p_{(1/v)}(\rr {2d}),
$$
for some $p\in (0,\infty ]$ and
some submultiplicative $v\in \mascP _E(\rr {2d})$.
\end{rem}

\par

In the following proposition we list some basic
properties for modulation spaces.

\par

\begin{prop}\label{Prop:ModNormInv}
Let $\omega \in \mascP _E(\rr {2d})$,
and let $\mascB$ be an invariant QBF space on $\rr {2d}$
of order $p\in (0,1]$. Then the following is true:
\begin{enumerate}
\item  $M(\omega ,\mascB)$ and $M_0(\omega ,\mascB)$
are independent
of the choices of $\phi \in \Sigma _1(\rr d)\setminus 0$ in
\eqref{Eq:ModNorm}, and different
choices of $\phi$ give rise to equivalent quasi-norms;

\vrum

\item if $M(\omega ,\mascB)$ is normal, then
$M_0(\omega ,\mascB)$ is normal, and $M(\omega ,\mascB)$
and $M_0(\omega ,\mascB)$
are quasi-Banach spaces of order $p$.
\end{enumerate}
\end{prop}

\par

A proof of Proposition \ref{Prop:ModNormInv} is available for
certain choices of $\mascB$ and window functions
$\omega$ and $v_0$ in e.{\,}g.
\cite{Fei1,FeiGro1,GaSa,Gro2,Rau1,Rau2}.
For general choices of  $\mascB$ and window functions,
we may argue in similar ways. In order to be self-contained
we present a proof of Proposition \ref{Prop:ModNormInv}
in Appendix \ref{App:B}.

%
%
%
%
%

\par

The next result shows the role of $\omega$ concerning
the magnitude of $M(\omega ,\mascB )$.

\par

\begin{prop}\label{Prop:ModEmb}
Let $\omega _1,\omega _2\in \mascP _E(\rr {2d})$,
$\mascB$ be an invariant 
QBF-space on $\rr {2d}$ with respect to $v\in
\mascP _E(\rr {2d})$.
Then the following is true:
\begin{enumerate}
\item if
${\omega_2}\lesssim {\omega_1}$,
then
\begin{align*}
M_0(\omega_1,\mascB)
\subseteq
M_0(\omega_2, \mascB),
\quad \text{and}\quad
M(\omega_1,\mascB)
&\subseteq
M(\omega_2, \mascB),
\end{align*}
and the injections
\begin{equation}\label{Eq:ModEmb}
\begin{aligned}
i\! : M_{0}(\omega _1,\mascB ) \to
M_{0}(\omega _2,\mascB )
\quad \text{and}\quad
i\! : M(\omega _1,\mascB ) \to
M(\omega _2,\mascB )
\end{aligned}
\end{equation}
are continuous;

\vrum

\item if in addition $v$ is bounded,
then the following conditions are equivalent:
\begin{itemize}
\item at least one of the mappings in
\eqref{Eq:ModEmb} is a continuous
injection;

\vrum

\item all the mappings in
\eqref{Eq:ModEmb} are continuous
injections;

\vrum

\item 
${\omega_2}\lesssim {\omega_1}$.
\end{itemize}
\end{enumerate}
\end{prop}

\par

Proposition \ref{Prop:ModEmb}
slightly generalizes parts of 
\cite[Theorem 3.7]{PfeTof}.
In order to be self-contained,
we give a proof in Appendix
\ref{App:B}.

\par

For compactness we have the following
partially generalization of \cite[Theorem 3.9]{PfeTof}.
The result follows
by similar arguments as in the proof
of \cite[Theorem 3.9]{PfeTof}. The details are left for the
reader.

\par

\begin{prop}\label{Prop:ModEmbCompact}
Let $\omega _1,\omega _2\in \mascP _E(\rr {2d})$,
and let $\mascB$ be an invariant 
QBF-space on $\rr {2d}$ with respect to
$v(x,\xi )=1$ everywhere.
Then the following conditions are equivalent:
\begin{itemize}
\item at least one of the mappings in
\eqref{Eq:ModEmb} is compact;

\vrum

\item all the mappings in
\eqref{Eq:ModEmb} are compact;

\vrum

\item 
$\displaystyle{\lim _{|(x,\xi )|\to \infty}
\frac {\omega_2(x,\xi )}{\omega_1(x,\xi )} =0}$.
\end{itemize}
\end{prop}

\par

\begin{rem}
If $\omega \in \mascP _E(\rr {2d})$, then
\begin{alignat}{2}
M_0 (\omega ,L^{p,q}(\rr {2d}))
&\subseteq
M(\omega ,L^{p,q}(\rr {2d}))
=
M^{p,q}_{(\omega )}(\rr d), &
\quad
p,q &\in (0,\infty ],
\label{Eq:EmbClassicalMod1}
\intertext{and}
M_0(\omega ,L^{p,q}(\rr {2d}))
&=
M(\omega ,L^{p,q}(\rr {2d}))
=
M^{p,q}_{(\omega )}(\rr d), &
\quad
p,q &\in (0,\infty ).
\label{Eq:EmbClassicalMod2}
\end{alignat}
In fact,
\eqref{Eq:EmbClassicalMod1}
and the last two equalities in
\eqref{Eq:EmbClassicalMod2} follow from
the fact that $M^{p,q}_{(\omega )}(\rr d)$ is continuously
embedded in $M^\infty _{(\omega )}(\rr d)$ when
$p,q\in (0,\infty ]$. The first equality in
\eqref{Eq:EmbClassicalMod2} follows from the fact
that $\Sigma _1(\rr d)$ is dense in
$M^{p,q}_{(\omega )}(\rr d)$ when $p,q<\infty$.
We also observe that the inclusion in
\eqref{Eq:EmbClassicalMod1} is strict, if and only if
$p=\infty$ or $q=\infty$. 
(See e.{\,}g. \cite{GaSa,Gro2,Toft10}.)

\par

A broader family of modulation spaces
is obtained by choosing $\mascB$ in Definition
\ref{Def:ModSpaces} as the mixed quasi-Banach
Orlicz space $L^{\Phi ,\Psi }(\rr {2d})$,
where $\Phi$ and $\Psi$ are quasi-Young functions
(cf. e.{\,}g. \cite{SchFuh,TofUst,ToUsNaOz}).
In this case we obtain
\begin{equation}\label{Eq:OrlModSp}
M_0(\omega ,L^{\Phi ,\Psi}(\rr {2d}))
\subseteq 
M (\omega ,L^{\Phi ,\Psi}(\rr {2d})) 
=
M^{\Phi ,\Psi}_{(\omega )}(\rr d),
\end{equation}
where $M^{\Phi ,\Psi}_{(\omega )}(\rr d)$ is the
(quasi-Banach) Orlicz modulation space with
respect to $\Phi$, $\Psi$ and $\omega$.
If $p\in \mathbf R_+$,
$$
\Phi _p(t)=t^p
\quad \text{and}\quad
\Phi _\infty(t)
=
\begin{cases}
0, & t\le 1,
\\[1ex]
\infty , & t>1,
\end{cases}
$$
then $L^{\Phi _p}=L^p$, when $p\in (0,\infty ]$.
This in turn gives $M^{\Phi _p,\Phi _q}_{(\omega )}
=M^{p,q}_{(\omega )}$. We observe that the modulation
spaces in \eqref{Eq:OrlModSp} are normal, in view of
\cite{ToUsNaOz}.
We refer to  \cite{SchFuh,TofUst,ToUsNaOz} for
more facts on Orlicz modulation spaces.
\end{rem}

\par

\begin{rem}
Let $\omega \in \mascP _E(\rr {2d})$ and $\mascB$ be
an invariant QBF space. We observe that the
only possible situation in which $M(\omega ,\mascB)$
fails to be normal is when $\mascB$ is not a Banach space
because if $\mascB$
in Definition \ref{Def:ModSpaces} is an
invariant BF space on $\rr {2d}$, then
$M(\omega ,\mascB )$
is normal. (See e.{\,}g. \cite{PfeTof}.)
\end{rem}

\par

\section{Convolution and twisted convolution
factorizations of
Wiener amalgam spaces}\label{sec2}

\par

In this section we apply Theorem \ref{Thm:ApprIdent}
from the previous section to show that Wiener
amalgam spaces of the form
$\WL ^{1,r}_{(v)}(\rr d)$
for $r\in (0,1]$ is a factorization algebra under 
convolution
and twisted convolution. Another application of
Theorem \ref{Thm:ApprIdent} then leads to that
$\WL ^{p,q}_{(\omega )}(\rr d)$
is a factorization
module over these algebras, when
$p\ge 1$, $q\ge r$ and $\omega$ being
$v$-moderate.

\par

\subsection{Convolution factorizations of Wiener amalgam spaces}

\par

%

Our main application of Theorem \ref{Thm:ApprIdent}
for convolutions on Wiener amalgam spaces is
Theorem \ref{Thm:SurjConvWien} below. In the
following preparing result we
ensure needed continuity properties for
Wiener amalgam spaces under convolution.
Here the involved weight functions should satisfy
\eqref{Eq:Moderate}, and the constant
\begin{align}
c_v
&=
\sup _{x\in Q}(v(x),1),\
\quad Q=[0,1]^d
\label{Eq:SubmultConst}
\end{align}
appears naturally.

\par

\begin{prop}\label{Prop:ConvWien}
Let $r\in (0,1]$, $p\in [1,\infty ]$, $q\in [r,\infty ]$ and
$\omega ,v\in \mascP _E(\rr d)$ be such that
$\omega$ is $v$-moderate. Then the map
$(f,g) \mapsto f*g$ from
$\Sigma _1(\rr d)\times \Sigma _1(\rr d)$ to
$\Sigma _1(\rr d)$ is uniquely
extendable to a continuous and commutative
map from
$$
\WL ^{1,r}_{(v)}(\rr d)
\times
\WL ^{p,q}_{(\omega )}(\rr d)
\quad \text{or}\quad
\WL ^{p,q}_{(\omega )}(\rr d)
\times
\WL ^{1,r}_{(v)}(\rr d)
$$
to $\WL ^{p,q}_{(\omega )}(\rr d)$.
If
$c_v$ is given by \eqref{Eq:SubmultConst},
then
\begin{equation}\label{Eq:WienProdEst}
\begin{aligned}
\nm {f*g}{\WL ^{p,q}_{(\omega )}}
&=
\nm {g*f}{\WL ^{p,q}_{(\omega )}}
\le 2^{d}c_v
\nm f{\WL ^{1,r}_{(v)}}
\nm g{\WL ^{p,q}_{(\omega )}}, 
\\[1ex]
f&\in \WL ^{1,r}_{(v)}(\rr d),\ 
g\in \WL ^{p,q}_{(\omega )}(\rr d).
\end{aligned}
\end{equation}
\end{prop}

\par

Proposition \ref{Prop:ConvWien}
essentially shows that
\begin{equation}\label{Eq:ConvWien}
\WL ^{1,r}_{(v)}(\rr d)
*
\WL ^{p,q}_{(\omega )}(\rr d)
=
\WL ^{p,q}_{(\omega )}(\rr d)
*
\WL ^{1,r}_{(v)}(\rr d)
\subseteq
\WL ^{p,q}_{(\omega )}(\rr d).
\end{equation}

\par

A proof of Proposition \ref{Prop:ConvWien}
can be found in e.{\,}g. \cite[Lemma 2.9]{GaSa} 
or \cite[Proposition 2.5]{Toft13}. (See also
\cite{Gro2,Rau1} for related results.)
In order to be self contained we present a
proof in Appendix \ref{App:C}.

\par

If in addition $p,q<\infty$,
then the following result shows that
\eqref{Eq:ConvWien} holds with
\emph{equality}.

\par

\begin{thm}\label{Thm:SurjConvWien}
Let $r\in (0,1]$, $p\in [1,\infty )$,
$q\in [r,\infty )$ and
$\omega ,v\in \mascP _E(\rr d)$ be such that
$\omega$ is $v$-moderate. Then
\begin{equation}
\WL ^{1,r}_{(v)}(\rr d)
*
\WL ^{p,q}_{(\omega )}(\rr d)
=
\WL ^{p,q}_{(\omega )}(\rr d)
*
\WL ^{1,r}_{(v)}(\rr d)
=
\WL ^{p,q}_{(\omega )}(\rr d).
\end{equation}
\end{thm}

\par

We notice that
Theorem \ref{Thm:SurjConvWien}
means that the mappings in
Proposition \ref{Prop:ConvWien}
are surjective
when $p,q<\infty$.

\par

\begin{proof}
Since the convolution is commutative
it suffices to prove that
$$
\WL ^{p,q}_{(\omega )}(\rr d)
*
\WL ^{1,r}_{(v)}(\rr d)
=
\WL ^{p,q}_{(\omega )}(\rr d).
$$

\par

The assertion will follow from Theorem 
\ref{Thm:ApprIdent} and
Proposition \ref{Prop:ConvWien}
if we show that $\maclB$ possess
a bounded (left) approximate
identity for $(\maclB ,\maclM)$, with
$$
\maclB = \WL ^{1,r}_{(v)}(\rr d)
\quad \text{and}\quad
\maclM = \WL ^{p,q}_{(\omega )}(\rr d).
$$
Hence, if $\phi \in C_0^\infty (\rr d;[0,1])$
satisfies $\int _{\rr d}\phi (y)\, dy=1$,
it suffices to prove
\begin{equation}\label{Eq:ApprUnitConvWien}
\nm {f-f*\phi _\ep}{\WL ^{p,q}_{(\omega )}}
\to 0
\quad \text{as}\quad
\ep \to 0+,
\quad \phi _\ep (x)
=
\ep ^{-d}\phi (\ep ^{-1}x),\ \ep >0 ,
\end{equation}
which is proven in similar ways as when
dealing with Lebesgue spaces.

\par

In fact, since $p,q<\infty$, it follows that 
$C_0^\infty (\rr d)$ is dense in
$\WL ^{p,q}_{(\omega )}(\rr d)$, which
implies that it suffices to
prove \eqref{Eq:ApprUnitConvWien} when $f\in C_0^\infty (\rr d)$.

\par

Therefore, suppose that  $f\in C_0^\infty (\rr d)$,
and let $Q=[0,1]^d$.
Then Minkowski's inequality and the mean-value
theorem give that for some $y_j=y_j(\ep )\in Q$, one has
\begin{align*}
\nm {f - f&*\phi _\ep}{\WL ^{p,q}_{(\omega )}}
\\[1ex]
&=
\left (
\sum _{j\in \zz d}\left (
\int _{j+Q} \left |
\int _{\rr d} (f(x)-f(x-\ep y))\phi (y)\, dy
\right | ^{p}\, dx
\right ) ^{\frac qp}\omega (j)^q
\right )^{\frac 1q}
\\[1ex]
&\le
\left (
\sum _{j\in \zz d} \left (
\int _{\rr d} \left (
\int _{j+Q} \left | (f(x)-f(x-\ep y))\omega (j)
\right | ^{p}\, dx
\right ) ^{\frac 1p}
\phi (y)\, dy
\right )^{q}
\right )^{\frac 1q}
\\[1ex]
&\asymp
\left (
\sum _{j\in \zz d} \left (
\int _{j+Q}
\left | (f(x)-f(x-\ep y_j(\ep )))\omega (j)
\right | ^{p}\, dx
\right ) ^{\frac qp}
\right )^{\frac 1q}
\left (
\int _{\rr d} \phi (y)\, dy
\right )
\end{align*}
which tends to $0$ as $\ep \to 0+$. Here
observe that at most finite numbers of terms
in the sums are non-zero.
This gives the result.
\end{proof}

\par

By letting $q=p<\infty$ in the previous result we get the following.

\par

\begin{cor}\label{Cor:SurjConvWien}
Let $r\in (0,1]$, $p\in [1,\infty )$ and
$\omega ,v\in \mascP _E(\rr d)$ be
such that $\omega$ is $v$-moderate holds.
Then
$$
\WL ^{1,r}_{(v)}(\rr {d})
*
L^p_{(\omega )}(\rr {d})
=
L^p_{(\omega )}(\rr {d})
*
\WL ^{1,r}_{(v)}(\rr {d})
=
L^p_{(\omega )}(\rr {d}).
$$
\end{cor}

\par

%
%

\par

\subsection{Twisted convolution factorizations
of Wiener amalgam spaces}

\par

A Fourier transform and non-commutative 
convolution multiplication,
connected to the Weyl product in the theory of
pseudo-differential operators concern
the symplectic Fourier transform and
the twisted convolution.
The
\emph{symplectic Fourier transform} of $\fka \in
\Sigma _1(\rr {2d})$ is defined by the formula
\begin{equation}
\label{Eq:SympFourTrans}
(\mascF _\sigma \fka ) (X)
=
\pi^{-d}\int _{\rr {2d}}\fka (Y)
e^{2 i \sigma(X,Y)}\,  dY.
\end{equation}
Here $\sigma$ from $\rr {2d}\times \rr {2d}$
to $\mathbf R$
is the symplectic form, given by
$$
\sigma(X,Y) = \scal y \xi - \scal x \eta ,
\qquad X=(x,\xi )\in \rr {2d},\ Y=(y,\eta )\in \rr {2d}.
$$

\par

We note that
$$
\mascF _\sigma = T\circ (\mascF \otimes (\mascF ^{-1})),
\quad \text{when}\quad
(T\fka )(x,\xi ) =2^da(2\xi ,2x).
$$
In particular, ${\mascF _\sigma}$ is continuous on
$\Sigma _1(\rr {2d})$, and extends uniquely to a 
homeomorphism on $\Sigma _1'(\rr {2d})$, and to a
unitary map on $L^2(\rr {2d})$, since similar
facts hold for $\mascF$. Furthermore, 
$\mascF _\sigma ^{2}$ is the identity
operator.

\par

Let
$\fka ,\fkb \in
\Sigma _1 (\rr {2d})$. Then the \emph{twisted
convolution} of $\fka$ and $\fkb$ is
defined by the formula
\begin{equation}\label{Eq:TwistConv}
(\fka  \ast _\sigma \fkb ) (X)
= \left ( {{\frac 2\pi}}
\right )^{\frac d2} \int _{\rr {2d}}
\fka (X-Y) \fkb (Y) e^{2 i \sigma(X,Y)}\, dY.
\end{equation}
The definition of $*_\sigma$ extends in
different ways. For example,
it extends to a continuous multiplication on
$L^p(\rr {2d})$ when $p\in
[1,2]$, and to a continuous map from
$\Sigma _1'(\rr {2d})\times
\Sigma _1 (\rr {2d})$ to $\Sigma _1'(\rr {2d})$.
We also remark that for the twisted convolution
we have
\begin{equation}\label{Eq:FourTwist}
\mascF _\sigma (\fka  *_\sigma \fkb )
=
(\mascF _\sigma \fka ) *_\sigma \fka
=
\check{\fka } *_\sigma (\mascF _\sigma \fkb ),
\end{equation}
where $\check{\fka}(X)=\fka (-X)$ (cf. 
\cite{Fol,Toft3,Toft12,Toft15}).

\par

By
$$
|\fka *_\sigma \fkb |\le
{\textstyle {\left (
\frac \pi 2
\right )^{\frac d2}}}
(|\fka |* |\fkb |),
$$
similar arguments as in the proof of
Proposition \ref{Prop:ConvWien}
give the following. The details are left
for the reader. Here the involved
weight functions should satisfy similar
properties as before. More precisely,
they should satisfy
\begin{align}
\omega (X+Y)
&\le
\omega (X)v(Y),\qquad X,Y\in \rr {2d},
\label{Eq:ModerateR2d}
\intertext{and the constant}
c_v
&=
\sup _{X\in Q}(v(X),1),\
\quad Q=[0,1]^{2d}
\label{Eq:SubmultConst2}
\end{align}
is still important.

\par

\begin{prop}\label{Prop:TwistConvWien}
Let $r\in (0,1]$, $p\in [1,\infty ]$,
$q\in [r,\infty ]$ and
$\omega ,v\in \mascP _E(\rr {2d})$ be such that
\eqref{Eq:ModerateR2d} holds.
Then the map $(\fka ,\fkb )
\mapsto
\fka *_\sigma \fkb$ from
$\Sigma _1(\rr {2d})\times \Sigma _1(\rr {2d})$ to
$\Sigma _1(\rr {2d})$ is uniquely
extendable to continuous mappings from
$$
\WL ^{1,r}_{(v)}(\rr {2d})
\times
\WL ^{p,q}_{(\omega )}(\rr {2d})
\quad \text{or}\quad
\WL ^{p,q}_{(\omega )}(\rr {2d})
\times
\WL ^{1,r}_{(v)}(\rr {2d})
$$
to $\WL ^{p,q}_{(\omega )}(\rr {2d})$. If
$c_v$ is given by \eqref{Eq:SubmultConst2}, then
\begin{equation}
\label{Eq:WienTwistProdEst}
\begin{aligned}
\max \big (
\nm {\fka *_\sigma \fkb }
{\WL ^{p,q}_{(\omega )}},
\nm {\fkb *_\sigma \fka }
{\WL ^{p,q}_{(\omega )}} \big )
&\le 4^{d}{\textstyle {\left (
\frac \pi 2
\right )^{\frac d2}}}
c_v
\nm {\fka}{\WL ^{1,r}_{(v)}}
\nm {\fkb}{\WL ^{p,q}_{(\omega )}}, 
\\[1ex]
\fka &\in \WL ^{1,r}_{(v)}(\rr {2d}),\ 
\fkb \in \WL ^{p,q}_{(\omega )}(\rr {2d}).
\end{aligned}
\end{equation}
\end{prop}

\par

The following result shows that if
$p,q<\infty$, then the mappings in
Proposition \ref{Prop:TwistConvWien}
are \emph{surjective.}
The result follows by similar arguments
as in the proof of Theorem \ref{Thm:SurjConvWien},
using Proposition \ref{Prop:TwistConvWien}
instead of Proposition \ref{Prop:ConvWien}.
The details are left for the reader.

\par

\begin{thm}\label{Thm:SurjTwistConvWien}
Let $r\in (0,1]$, $p\in [1,\infty )$,
$q\in [r,\infty )$ and
$\omega ,v\in \mascP _E(\rr {2d})$ be such that
$\omega$ is $v$-moderate. Then
\begin{equation}
\WL ^{1,r}_{(v)}(\rr {2d})
*_\sigma 
\WL ^{p,q}_{(\omega )}(\rr {2d})
=
\WL ^{p,q}_{(\omega )}(\rr {2d})
*_\sigma 
\WL ^{1,r}_{(v)}(\rr {2d})
=
\WL ^{p,q}_{(\omega )}(\rr {2d}).
\end{equation}
\end{thm}

\par

By letting $q=p$ in the previous result we
get the following.

\par

\begin{cor}\label{Cor:SurjTwistConvWien}
Let $r\in (0,1]$, $p\in [1,\infty )$ and
$\omega ,v\in \mascP _E(\rr {2d})$ be
such that \eqref{Eq:ModerateR2d} holds.
Then
$$
\WL ^{1,r}_{(v)}(\rr {2d})
*_\sigma 
L^p_{(\omega )}(\rr {2d})
=
L^p_{(\omega )}(\rr {2d})
*_\sigma 
\WL ^{1,r}_{(v)}(\rr {2d})
=
L^p_{(\omega )}(\rr {2d}).
$$
\end{cor}

\par

%
%

\par

\begin{rem}
Suppose that $p\in [1,\infty )$. Since
$$
\WL ^{1,r}_{(v)}(\rr d) \subsetneq L^1_{(v)}(\rr d)
$$
when $r\in (0,1)$ and $\omega ,v\in \mascP _E(\rr d)$
are such that $\omega$ is $v$-moderate,
it follows that Corollary \ref{Cor:SurjConvWien}
is a strict improvement of the weighted Rudin's identity
$$
L^1_{(v)}(\rr {d})
*
L^p_{(\omega )}(\rr {d})
=
L^p_{(\omega )}(\rr {d})
*
L^1_{(v)}(\rr {d})
=
L^p_{(\omega )}(\rr {d}),
$$
which follows by choosing $r=1$ in that corollary. In the
same way it follows that Corollary
\ref{Cor:SurjTwistConvWien}
is a strict improvement of
$$
L^1_{(v)}(\rr {2d})
*_\sigma 
L^p_{(\omega )}(\rr {2d})
=
L^p_{(\omega )}(\rr {2d})
*_\sigma 
L^1_{(v)}(\rr {2d})
=
L^p_{(\omega )}(\rr {2d})
$$
when $\omega ,v\in \mascP _E(\rr {2d})$
satisfy that $\omega$ is $v$-moderate.
(See also
\eqref{Eq:RudinIdentImproved} in the
introduction.)
\end{rem}

\par

\begin{rem}\label{Rem:NoFactCond}
Suppose that $p=\infty$.
Then we observe that the assertion
in Theorem \ref{Thm:SurjConvWien}
fails to hold. In fact, using that
$\Sigma _1(\rr d)$ is dense in
$\WL ^{1,r}_{(v)}(\rr d)$
it follows that
$$
\WL ^{1,r}_{(v)}(\rr {d})
*
\WL ^{\infty ,q}_{(\omega )}(\rr {d})
=
\WL ^{\infty ,q}_{(\omega )}(\rr {d})
* 
\WL ^{1,r}_{(v)}(\rr {2d})
\subseteq
C(\rr d),
$$
while
$$
\WL ^{\infty ,q}_{(\omega )}(\rr {d})
\nsubseteq
C(\rr d).
$$
In similar ways it follows that
the assertion in Theorem 
\ref{Thm:SurjTwistConvWien} 
fails to hold when $p=\infty$.

\par

For the moment we are not able to
draw any conclusion wether
the assertions in Theorems 
\ref{Thm:SurjConvWien}
and
\ref{Thm:SurjTwistConvWien}
fail to hold when $p<\infty$ and $q=\infty$.
\end{rem}

\par

\section{Convolution factorizations of
modulation spaces}\label{sec3}

\par

In this section we apply Theorem \ref{Thm:ApprIdent}
from Section \ref{App:A} to show that with natural
assumptions on $\omega$, $v$ and $\mascB$,
the modulation space $M(\omega ,\mascB)$ is
a factorization module over the Wiener
amalgam space $\WL ^{1,r}_{(v)}(\rr d)$
for $r\in (0,1]$ under convolution. In particular
it follows that if $\omega$ is $v$-moderate and
$r\le p,q<\infty$, then $M^{p,q}_{(\omega )}(\rr d)$
and $W^{p,q}_{(\omega )}(\rr d)$ are
factorization modules over
$\WL ^{1,r}_{(v)}(\rr d)$.

\par

In order to be more specific, we have the following.

\par

\begin{prop}\label{Prop:ConvLebMod}
Suppose 
$\omega ,v_1\in \mascP _E(\rr {2d})$
are such that $v_1$ is submultiplicative
and $\omega$ is $v_1$-moderate, $\mascB$
is an invariant
QBF space on $\rr {2d}$ with respect
to $v_0\in \mascP _E(\rr {2d})$ of order
$r\in (0,1]$, and that $M(\omega ,\mascB)$
is normal. Also let
$$
v(x)=v_1(x,0)v_0(x,0).
$$
Then the map $(f,g)\mapsto f*g$ from
$\Sigma _1(\rr d)\times \Sigma _1(\rr d)$
to $\Sigma _1(\rr d)$ is uniquely extendable
to a continuous map from
$$
\WL ^{1,r}_{(v)}(\rr d)
\times
M(\omega ,\mascB)
\quad \text{or}\quad
M(\omega ,\mascB)
\times
\WL ^{1,r}_{(v)}(\rr d)
$$
to $M(\omega ,\mascB)$. 
Furthermore,
\begin{equation}
\label{Eq:ConvLebModEst}
\begin{aligned}
\nm {f*g}{M(\omega ,\mascB )}
&=
\nm {g*f}{M(\omega ,\mascB )}
\lesssim
\nm f{\WL ^{1,r}_{(v)}}
\nm g{M(\omega ,\mascB )},
\\[1ex]
f&\in \WL ^{1,r}_{(v)}(\rr d),
\quad 
g\in M(\omega ,\mascB ).
\end{aligned}
\end{equation}
\end{prop}

\par

We observe that Proposition
\ref{Prop:ConvLebMod} essentially
means that
\begin{equation}\label{Eq:ConvLebModLeb}
\WL ^{1,r}_{(v)}(\rr d)
*
M(\omega ,\mascB)
=
M(\omega ,\mascB)
*
\WL ^{1,r}_{(v)}(\rr d)
\subseteq
M(\omega ,\mascB).
\end{equation}

\par

\begin{proof}
We replace the weights $\omega$, $v$, $v_0$ and
$v_1$ with equivalent smooth weights, which is
possible in view of e.{\,}g. \cite{Toft12}.
First suppose that
$f\in \Sigma _1(\rr d)$
and $g\in M(\omega ,\mascB )$.
Let $Q=[0,1]^{d}$
$$
G(x,\xi )
=
|\maclT _\phi g(x,\xi )\omega (x,\xi )|
\quad \text{and}\quad
H(x,\xi )
=
\int _{\rr d} |f(y)v_1(y,0)|
G(x-y,\xi )\, dy.
$$
Then $G$ and $H$ are continuous, and
\eqref{Eq:TOpConv} and the fact that 
$\omega$ is $v_1$ moderate give
$$
|(\maclT _\phi (f*g))(x,\xi )\omega (x,\xi )|
\lesssim H(x,\xi ).
$$
By
the mean-value theorem, Minkowski's inequality,
\eqref{Eq:TranslInv} and
\eqref{Eq:TOpConv} give for some $y_j\in Q$, $j\in \zz d$, that
\begin{align*}
\nm {f*g}{M(\omega ,\mascB )}^r
&\lesssim
\nm {H}{\mascB}^r
=
\NM {\sum _{j\in \zz {d}}
\int _{j+Q}|f(y)v_1(y,0)|
G\big (\cdo -(y,0) \big )\, dy}
{\mascB}^r
\\[1ex]
&=
\NM {\sum _{j\in \zz {d}}
G\big (\cdo -(j+y_j,0) \big )
\int _{j+Q}|f(y)v_1(y,0)|\, dy }{\mascB}^r
\\[1ex]
&\le
\sum _{j\in \zz {d}}
\NM {G\big ( \cdo -(j+y_j,0) \big )}
{\mascB}^r
\left ( \int _{j+Q}|f(y)v_1(y,0)|\, dy \right )^r
\\[1ex]
&\lesssim
\sum _{j\in \zz {d}}
\left (
\int _{j+Q}|f(y)v_1(y,0)v_0(j+y_j,0)|
\, dy
\right )^r\nm {G}{\mascB}^r
\\[1ex]
&\asymp
\sum _{j\in \zz {d}}
\left ( \int 
_{j+Q}|f(y)v_1(y,0)v_0(y,0)|\, dy
\right )^r\nm {G}{\mascB}^r
\\[1ex]
&= 
\nm f{\WL ^{1,r}_{(v)}}^r
\nm {G}{\mascB}^r,
\end{align*}
which shows that
\begin{equation*}
\nm {f*g}{M(\omega ,\mascB )}
\lesssim
\nm f{\WL ^{1,r}_{(v)}},
\nm {g}{M(\omega ,\mascB )}
\end{equation*}
and \eqref{Eq:ConvLebModEst}
follows in the
case when $f\in \Sigma _1(\rr d)$.
The assertion now follows in general 
from
the facts that \eqref{Eq:ConvLebModEst}
holds with $f\in \Sigma _1(\rr d)$
and that $\Sigma _1(\rr d)$
is dense in $\WL ^{1,r}_{(v)}(\rr d)$.
\end{proof}

\par

\begin{rem}
Evidently, Proposition \ref{Prop:ConvLebMod}
holds true with $M_0(\omega ,\mascB)$
in place of $M_0(\omega ,\mascB)$ at each
occurrence. In particular, 
\begin{equation}\tag*{(\ref{Eq:ConvLebModLeb})$'$}
\WL ^{1,r}_{(v)}(\rr d)
*
M_0(\omega ,\mascB)
=
M_0(\omega ,\mascB)
*
\WL ^{1,r}_{(v)}(\rr d)
\subseteq
M_0(\omega ,\mascB).
\end{equation}
\end{rem}

\par

The following result shows that
\eqref{Eq:ConvLebModLeb}$'$ holds with
equality.

\par

\begin{thm}\label{Thm:ConvLebMod}
$\omega ,v_1\in \mascP _E(\rr {2d})$
are such that $v_1$ is submultiplicative
and $\omega$ is $v_1$-moderate, and suppose
that $\mascB$
is an invariant
QBF space on $\rr {2d}$ with respect
to $v_0\in \mascP _E(\rr {2d})$ of order
$r\in (0,1]$. Also let
$$
v(x)=v_1(x,0)v_0(x,0).
$$
Then
\begin{equation}\tag*{(\ref{Eq:ConvLebModLeb})$''$}
\WL ^{1,r}_{(v)}(\rr d)
*
M_0(\omega ,\mascB)
=
M_0(\omega ,\mascB)
*
\WL ^{1,r}_{(v)}(\rr d)
=
M_0(\omega ,\mascB).
\end{equation}
\end{thm}

\par

\begin{proof}
Since $\Sigma _1(\rr d)$ is dense in
$M_0(\omega ,\mascB)$, it follows
from Theorem \ref{Thm:ApprIdent}
that the result follows if we prove that
\begin{equation}\label{Eq:ApprModNorm}
\lim _{\ep \to 0}\nm {f*\psi _\ep -f}
{M(\omega ,\mascB)}=0,\qquad
f\in \Sigma _1(\rr d),
\end{equation}
when $\psi _\ep =\ep ^{-d}\psi (\ep ^{-1}\cdo )$,
where $\psi \in \Sigma _1(\rr d)$ satisfies
$$
\int _{\rr d}\psi (x)\, dx=1.
$$

\par

For $f\in \Sigma _1(\rr d)$ and 
$v_2=v\cdot v_1$ we have
\begin{align*}
\nm {f(\cdo -x)e^{i\scal \cdo \xi }}
{M(\omega ,\mascB )}
&\asymp
\nm {V_\phi f(\cdo -(x,\xi ))\cdot \omega}{\mascB}
\\[1ex]
&\lesssim
\nm {V_\phi f(\cdo -(x,\xi ))
\cdot \omega (\cdo -(x,\xi ))}{\mascB}v_1(x,\xi )
\\[1ex]
&\lesssim
\nm {V_\phi f
\cdot \omega}{\mascB}v(x,\xi )v_1(x,\xi )
=
\nm f{M(\omega ,\mascB )}v_2(x,\xi ).
\end{align*}
By Feichtinger's
minimization principle for translation and modulation
invariant quasi-Banach spaces it follows that
$M^{r,r}_{(v_2)}(\rr d)\subseteq M(\omega ,\mascB)$
and that
$$
\nm f{M(\omega ,\mascB)}
\lesssim 
\nm f{M^{r,r}_{(v_2)}},
\qquad
f\in M^{r,r}_{(v_2)}(\rr d).
$$
(See \cite[Theorem 2.4]{Toft18}.)
This gives
\begin{gather*}
\nm {f*\psi _\ep -f}{M(\omega ,\mascB)}^r
\lesssim
\nm {f*\psi _\ep -f}{M^{r,r}_{(v_2)}}^r
\\[1ex]
\lesssim
\iint _{\rr {2d}}
\left ( \int _{\rr d}
|V_\phi f(x-\ep y,\xi )-V_\phi f(x,\xi )|
\, |\psi (y)|v_2(x,\xi )\, dy \right )^r\, dxd\xi 
\end{gather*}
Since $v_2\in \mascP _E(\rr {2d})$,
$V_\phi f\in \Sigma _1(\rr {2d})$ and
$\psi \in \Sigma (\rr d)$, it follows that
for some $r_0\ge 0$ we have
\begin{align*}
v_2(x,\xi )
&\lesssim
e^{r_0(|x|+|\xi |)},
\\[1ex]
|V_\phi f(x,\xi )|
&\lesssim
e^{-r(|x|+|\xi |)}
\intertext{and}
|\psi (x)|
&\lesssim
e^{-r|x|},
\end{align*}
for every $r>0$. This implies that
\begin{multline*}
H_\ep (x,y,\xi ) \equiv
|V_\phi f(x-\ep y,\xi )-V_\phi f(x,\xi )|
\, |\psi (y)|v_2(x,\xi )
\\[1ex]
\le Ce^{-r(|x|+|y|+|\xi |)}\in L^1(\rr {3d}),
\end{multline*}
where the constant $C>0$ is independent of
$\ep \in (0,1]$.
Since $H_\ep (x,y,\xi )$ tends to $0$
pointwise as $\ep \to 0+$,
Lebesgue's theorem gives
$$
0\le
\nm {f*\psi _\ep -f}{M(\omega ,\mascB)}^r
\lesssim
\iint _{\rr {2d}}
\left ( \int _{\rr d}
H_\ep (x,y,\xi )\, dy \right )^r\, dxd\xi 
\to 0 
$$
as $\ep \to 0+$. This gives
\eqref{Eq:ApprModNorm} and the result
follows.
\end{proof}

%
%

\par

\begin{cor}
Suppose $p,q,r\in (0,\infty )$ satisfy $r\le \min (p,q,1)$, and let $\omega ,v
\in \mascP _E(\rr {2d})$ be such that $\omega$
is $v$-moderate.
Then
$$
\WL ^{1,r}_{(v)}(\rr d)
*M^{p,q}_{(\omega )}(\rr d)
=
M^{p,q}_{(\omega )}(\rr d)
$$
and
$$
\WL ^{1,r}_{(v)}(\rr d)
*W^{p,q}_{(\omega )}(\rr d)
=
W^{p,q}_{(\omega )}(\rr d).
$$
\end{cor}

\par

By choosing $r=1$ in the previous corollary we obtain
the following.

\begin{cor}
Suppose $p,q\in [1,\infty )$, and let $\omega ,v
\in \mascP _E(\rr {2d})$ be such that $\omega$
is $v$-moderate.
Then
$$
L^1_{(v)}(\rr {d})
*M^{p,q}_{(\omega )}(\rr d)
=
M^{p,q}_{(\omega )}(\rr d)
$$
and
$$
L^1_{(v)}(\rr {d})
*W^{p,q}_{(\omega )}(\rr d)
=
W^{p,q}_{(\omega )}(\rr d).
$$
\end{cor}

\par

\section{Convolution and twisted
convolution factorizations of Schatten-von Neumann
symbols in pseudo-differential calculus}\label{sec4}

\par

In this section we first recall 
in Subsections
\ref{subsec5.1}--\ref{subsec5.3}
some facts about
pseudo-differential operators, their 
Schatten-von Neumann symbol classes and
properties on compositions of such operators
on symbol levels and their links to twisted
convolution. Thereafter we 
apply Theorem \ref{Thm:ApprIdent}
from Section \ref{App:A} to show that such
symbol classes are factorization modules
over the Wiener amalgam spaces
$\WL ^{1,r}(\rr {2d})$. In the last
part we apply Theorem \ref{Thm:ApprIdent}
to show that these symbol classes are
factorization modules over the space of
all symbols giving rise to compact operators.

%

\par

\subsection{Schatten-von Neumann classes}
\label{subsec5.1}

\par

The definition of Schatten-von Neumann classes
is often performed in background of singular
values (cf. e.{\,}g. 
\cite{Sim,BirSol,Toft12}).
Let $\maclH _1$ and $\maclH _2$ be
Hilbert spaces and let $T$ be a
linear operator from $\maclH  _1$ to
$\maclH _2$. The singular value of
$T$ of order $j\ge 1$ is defined as
$$
\sigma _j(T) = \sigma _j(T;\maclH _1 , \maclH _2 )
\equiv \inf \nm {T-T_0}{\maclH _1\to \maclH _2},
$$
where the infimum is taken over all linear
operators $T_0$ from $\maclH _1$ to
$\maclH _2$ of rank at most $j-1$.
The operator $T$ is said to be a Schatten-von
Neumann operator of order
$p\in (0,\infty ]$ if
\begin{equation}\label{SchattenNormBanach}
\nm T{\mascI _p(\maclH _1,\maclH _2)}
\equiv
\nm {\{\sigma _j(T)\}_{j=1}^\infty }
{\ell ^p(\mathbf Z_+)}
\end{equation}
is finite. The set of Schatten-von Neumann
operators from $\maclH _1$ to $\maclH _2$
of order $p\in (0,\infty ]$ is denoted by
$\mascI _p(\maclH _1,\maclH_2)$.

\par

We observe that if $p<\infty$,
then $\mascI _p(\maclH _1,\maclH _2)$
is contained in
$\maclK (\maclH _1,\maclH _2)$,
the set of compact 
operators from $\maclH _1$
to $\maclH _2$. Furthermore,
$\mascI _\infty (\maclH _1,\maclH _2)$
agrees with the set of linear bounded operators
from $\maclH _1$ to $\maclH _2$, with
equality in norms. We also have that
$\mascI _1(\maclH _1,\maclH _2)$
and $\mascI _2(\maclH _1,\maclH _2)$
are the Banach spaces which consist of
all trace-class and Hilbert-Schmidt operators,
respectively, from $\maclH _1$ to $\maclH _2$,
with equality in norms.

\par

It is evident that
$\mascI _p(\maclH _1,\maclH _2)$
is contained in the set of
linear and
compact operators from $\maclH _1$
to $\maclH _2$ when $p<\infty$.
If we set
$\mascI _\sharp (\maclH _1,\maclH _2)
\equiv \maclK (\maclH _1,\maclH _2)$,
and let $\mathbf R_\sharp
=\mathbf R_+\cup \{\sharp ,\infty \}$
with orderings
$$
x<\sharp ,\quad x\le \sharp ,\quad \sharp <\infty
\quad \text{and}\quad 
\sharp \le \infty
\quad \text{when}\quad
x\in \mathbf R_+,
$$
then we get
\begin{equation}
\begin{gathered}
\mascI _{p_1}(\maclH _1,\maclH _2)
\subseteq
\mascI _{p_2}(\maclH _1,\maclH _2)
\quad \text{and}\quad
\nm T{\mascI _{p_2}(\maclH _1,\maclH _2)}
\le
\nm T{\mascI _{p_1}(\maclH _1,\maclH _2)}
\\[1ex]
\text{when}\quad
p_1,p_2\in \mathbf R_\sharp ,
\quad
p_1\le p_2,
\quad 
T\in \mascI _\infty (\maclH _1,\maclH _2).
\end{gathered}
\end{equation}
Here recall that
\begin{equation}
\nm T{\mascI _\sharp (\maclH _1,\maclH _2)}
=
\nm T{\maclK (\maclH _1,\maclH _2)}
=
\nm T{\maclH _1\mapsto \maclH _2}
=
\nm T{\mascI _\infty (\maclH _1,\maclH _2)},
\quad
T\in \maclK (\maclH _1,\maclH _2)
\end{equation}

\par

It follows that if $p\in \mathbf R_\sharp$,
then $\mascI _{p}(\maclH _1,\maclH _2)$
is a quasi-Banach space of order $\min (1,p)$.
In particular,
$\mascI _{p}(\maclH _1,\maclH _2)$ is a Banach
space when $p\ge 1$. (See e.{\,}g. \cite{Sim}.)

\par

As a consequence of the spectral theorem
we have the following (see e.{\,}g. \cite{Sim}).
Here $\ON (\maclH)$ is the set of all orthonormal
sequences in the Hilbert space $\maclH$.
We also let $\ell ^\sharp (\mathbf Z_+)$
be the set of all sequences in
$\ell ^\infty (\mathbf Z_+)$ which tend to zero
at infinity.


\par

\begin{prop}\label{Prop:SpectralThm}
Let $p\in \mathbf R_\sharp$ with $p\neq \infty$,
$\maclH _1$, $\maclH _2$ be Hilbert spaces
and let
$T\in \mascI _\sharp (\maclH _1,\maclH _2)$.
Then the following is true:
\begin{enumerate}
\item for some
$\{ f_j\} _{j=1}^\infty \in \ON (\maclH _1)$
and
$\{ g_j\} _{j=1}^\infty \in \ON (\maclH _2)$
and unique non-negative decreasing sequence
$\{ \lambda _j\}_{j=1}^\infty$,
it holds
\begin{equation}\label{Eq:SpectralDecomp}
Tf = \sum _{j=1}^\infty
\lambda _j(f,f_j)_{\maclH _1}g_j,
\qquad f\in \maclH_1,
\end{equation}
where the sum converges with respect to
the $\maclH _2$ norm;

\vrum

\item if $\lambda _j$ are the same as in
\eqref{Eq:SpectralDecomp}, then
$\lambda _j=\sigma _j(T)$;

\vrum

\item $T\in \mascI _p (\maclH _1,\maclH _2)$,
if and only if
$\{ \lambda _j\} _{j=1}^\infty
\in \ell ^p(\mathbf Z_+)$, and
$$
\nm T{\mascI _p (\maclH _1,\maclH _2)}
=
\nm {\{ \lambda _j\} _{j=1}^\infty}
{\ell ^p (\mathbf Z_+)}
$$
\end{enumerate}
\end{prop}

\par

We also have the following H{\"o}lder's
inequality for compositions of
Schatten-von Neumann
classes. Again we refer to \cite{Sim}
for the proof.

\begin{prop}\label{Prop:SchattenHolder}
Let $p_0,p_1,p_2\in \mathbf R_\sharp$ be
such that
$$
\frac 1{p_0}
\le
\frac 1{p_1}+\frac 1{p_2}
\quad \text{and}\quad
(p_0,p_1,p_2)\neq (\sharp ,\infty , \infty),
$$
and let
$\maclH _j$, $j=0,1,2$, be Hilbert spaces.
Then the map $(T_1,T_2)\mapsto T_2\circ T_1$
restricts to a continuous map from
$$
\mascI _{p_1} (\maclH _0,\maclH _1)
\times
\mascI _{p_2} (\maclH _1,\maclH _2)
\quad \text{to}\quad
\mascI _{p_0} (\maclH _0,\maclH _2),
$$
and
\begin{multline*}
\nm {T_2\circ T_1}
{\mascI _{p_0} (\maclH _0,\maclH _2)}
\le
\nm {T_1}
{\mascI _{p_1} (\maclH _0,\maclH _1)}
\nm {T_2}
{\mascI _{p_2} (\maclH _1,\maclH _2)},
\\[1ex]
T_1\in \mascI _{p_1} (\maclH _0,\maclH _1),
\ 
T_2\in \mascI _{p_2} (\maclH _1,\maclH _2).
\end{multline*}
\end{prop}

\par

\subsection{Basic properties for
pseudo-differential operators and
Schatten-von Neumann symbols}
\label{subsec5.2}

\par

Recall that for any matrix $A\in \GL(d,\mathbf R)$,
and $\fka \in \Sigma _1'(\rr {2d})$, the
pseudo-differential operator $\op _A(\fka )$ is
the linear and continuous
mapping from $\Sigma _1(\rr d)$ to
$\Sigma _1'(\rr d)$, given by the formula
\begin{equation}\label{Eq:DefPseudo}
\op _A(\fka )f(x)
\equiv
(2\pi )^{-{2d}}\iint _{\rr {2d}}\fka
(x-A(x-y),\xi )f(y)
e^{i\scal {x-y}\xi}\, dyd\xi .
\end{equation}
Here the integrals shall be interpreted in
distribution sense when $\fka$ fails to belong
to $L^1(\rr {2d})$. That is, for general
$\fka \in \Sigma _1'(\rr {2d})$,
\eqref{Eq:DefPseudo} should be interpreted
as
\begin{equation}\tag*{(\ref{Eq:DefPseudo})$'$}
\begin{aligned}
\op _A(\fka )f(x)
&\equiv
(2\pi )^{-\frac d2}\scal {K_{\fka ,A}(x,\cdo )}f
\quad \text{with}
\\[1ex]
K_{\fka ,A}(x,y)
&\equiv
(\mascF _2^{-1}\fka )(x-A(x-y),x-y),
\end{aligned}
\end{equation}
where $\mascF _2F$ is the partial Fourier
transform of $F(x,y)$ with respect to the
$y$-variable. We observe that the map
$\fka \mapsto K_{\fka ,A}$ is homeomorphic
on
$$
\Sigma _1(\rr {2d}),\quad
\mascS (\rr {2d}),\quad
\mascS '(\rr {2d}),\quad
\quad \text{and on}\quad
\Sigma _1'(\rr {2d}),
$$
since similar facts hold true
for $\mascF _2$ and the map
$$
F(x,y)\mapsto F(x-A(x-y),x-y).
$$

\par

We observe that if $A=0$, then
$\op _A(\fka )=\op _0(\fka )$ is the
same as normal or
Kohn-Nirenberg representation
$\fka (x,D)$ of $\fka \in \Sigma _1'(\rr {2d})$.
If instead $A=\frac 12\cdot I_d$, where
$I_d$ is the $d\times d$ unit matrix, then
$\op _A(\fka )$ is the
Weyl quantization $\op ^w(\fka )$
of $\fka \in \Sigma _1'(\rr {2d})$.

\par

A combination of the previous homeomorphisms 
and Schwartz kernel theorem shows that
the map
\begin{equation}\label{Eq:SympPsDOMap}
\fka \mapsto \op _A(\fka )
\end{equation}
is bijective from $\Sigma _1'(\rr {2d})$
and the set of linear and continuous operators
from $\Sigma _1(\rr d)$ to $\Sigma _1'(\rr d)$.
The same holds true with $\mascS$ in place of
$\Sigma _1$ at each occurrence.
In particular it follows that if
$A,B\in \GL (d,\mathbf R)$ and
$\fka \in \Sigma _1'(\rr {2d})$, then there is 
a unique $\fkb \in \Sigma _1'(\rr {2d})$ such that
$\op _A(\fka )=\op _B(\fkb )$. By
straight-forward applications
of Fourier's inversion formula it follows that
\begin{equation}\label{Eq:CalculiTransfer}
\op _A(\fka )=\op _B(\fkb )
\qquad \Leftrightarrow \qquad
e^{i\scal {AD_\xi}{D_x}}\fka = e^{i\scal {BD_\xi}{D_x}}\fkb .
\end{equation}
(See also e.{\,}g. \cite{Toft15}.)

\par

If more restrictive $\fka \in \Sigma _1(\rr {2d})$,
then $\op _A(\fka )$ is continuous
from $\Sigma _1(\rr d)$ to $\Sigma _1(\rr d)$,
and is uniquely extendable to a continuous map
from $\Sigma _1'(\rr d)$ to $\Sigma _1(\rr d)$.

\medspace

For $p\in \mathbf R_\sharp$,
we let $s_{p,A}(\rr {2d})$
be the set of all $\fka \in \Sigma _1'(\rr {2d})$
such that $\op _A(\fka )$ belongs to
$\mascI _p=\mascI _p(L^2(\rr d))$.
We equip $s_{p,A}(\rr {2d})$ with the quasi-norm
$$
\nm \fka {s_{p,A}} \equiv
\begin{cases}
\nm {\op _A(\fka )}{\mascI _p}, & p\neq \sharp
\\[1ex]
\nm {\op _A(\fka )}{\mascI _\infty}, & p=\sharp ,
\end{cases}
$$
when $\fka \in s_{p,A}(\rr {2d})$.

\par

It follows from the kernel theorem by
Schwartz that the map \eqref{Eq:SympPsDOMap}
is an isometric isomorphism from
$s_{p,A}(\rr {2d})$
to $\mascI _p$. In particular,
$s_{p,A}(\rr {2d})$ is a quasi-Banach
space under the quasi-norm $\nm \cdo {s_{p,A}}$,
which is a Banach space when $p\ge 1$.
Since $\mascI _p$ is strictly increasing with
$p$ we get
$$
s_{p_1,A}(\rr {2d})
\subsetneq
s_{p_2,A}(\rr {2d}),
\quad \text{when}\quad
p_1,p_2\in \mathbf R_\sharp ,\ 
p_1<p_2.
$$
We also have that
\begin{alignat*}{2}
&\op _A(s_{1,A}(\rr {2d})), &
\quad
&\op _A(s_{2,A}(\rr {2d})),
\\[1ex]
&\op _A(s_{\sharp ,A}(\rr {2d})) &
\quad \text{and}\quad
&\op _A(s_{\infty ,A}(\rr {2d}))
\end{alignat*}
are the sets of trace-class, Hilbert-Schmidt, 
compact and continuous
operators, respectively, on $L^2(\rr d)$.

\par


In the remaining part of the subsection
we describe some Schatten-von Neumann
properties given in \cite{Sim} in
the framework of symbol calculus
for pseudo-differential operators.
We mainly follow the approaches in
\cite{Toft3,Toft15}.

\par

The spectral resolution of elements in
$s_p^A(\rr {2d})$ can be described in
terms of $A$-Wigner distributions. For
any $A\in \GL (d,\mathbf R)$, the
$A$-Wigner distribution of
$f,g\in \Sigma _1'(\rr d)$ is defined by
$$
W_{f,g}^A(x,\xi )
\equiv
\mascF \big ( f(x+A\cdo )
\overline{g(x+(A-I)\cdo )}\big )(\xi ),
$$
which makes sense as an element in
$\Sigma _1'(\rr {2d})$. If
$f,g\in \Sigma _1(\rr d)$,
then $W_{f,g}^A(x,\xi )$ is given by
$$
W_{f,g}^A(x,\xi )= (2\pi )^{-\frac d2}
\int _{\rr d}f(x+Ay)\overline{g(x+(A-I)y )}
e^{-i\scal y\xi}\, dy,
$$
which becomes an element in $\Sigma _1(\rr {2d})$.

\par


The following proposition describes fundamental links
between Wigner distributions and pseudo-differential
operators. Here the symbol product $\fka _1\wpr _A\fka _2$
of $\fka _1,\fka _2\in \Sigma _1'(\rr {2d})$ is defined by
the formula
\begin{equation}\label{Eq:SymbProd}
\op _A(\fka _1\wpr _A\fka _2 )
=
\op _A(\fka _1) \circ \op _A(\fka _2 ),
\end{equation}
provided the right-hand side makes sense
as a continuous
operator from $\Sigma _1(\rr d)$ to
$\Sigma _1'(\rr d)$.
The result follows by straight-forward
computations. The details are left
for the reader. (See also
\cite{Toft12,Toft15}.)

\par

\begin{prop}\label{Prop:WignerPseudoLinks}
Let $A\in \GL (d,\mathbf R)$
and $\fka \in \Sigma _1'(\rr {2d})$.
Then the following is true:
\begin{enumerate}
\item if $f,g\in \Sigma _1 (\rr d)$
then
$$
(\op _A(\fka )f,g)_{L^2(\rr {2d})}
=
(2\pi )^{-\frac d2}(\fka ,W_{g,f}^A)_{L^2(\rr {2d})}
\text ;
$$

\vrum

\item if $f\in \Sigma _1 (\rr d)$
and $g\in \Sigma _1'(\rr d)$, then
$$
\fka \wpr _A W_{f,g}^A = W_{\op _A(\fka )f,g}^A
\text ;
$$

\vrum

\item if $f\in \Sigma _1 (\rr d)$
and $f_0,g\in \Sigma _1'(\rr d)$, then
$$
\op _A(W_{f_0,g}^A)f = (2\pi )^{-\frac d2}(f,g)_{L^2}f_0.
$$
\end{enumerate}
\end{prop}

\par

A combination of \eqref{Eq:CalculiTransfer}
and Proposition \ref{Prop:WignerPseudoLinks}
gives
\begin{equation}\label{Eq:WignerTransfer}
e^{i\scal {AD_\xi}{D_x}}W_{f,g}^A
=
e^{i\scal {BD_\xi}{D_x}}W_{f,g}^B,
\qquad
f,g\in \Sigma _1'(\rr d).
\end{equation}

\par

We also observe that by straight-forward combination
of \eqref{Eq:CalculiTransfer} and \eqref{Eq:SymbProd}
we get
\begin{multline}\label{Eq:SymbProdTransf}
\big (
e^{-i\scal {AD_\xi}{D_x}}\fkc _1 
\big )
\wpr _A
\big (
e^{-i\scal {AD_\xi}{D_x}}\fkc _2 
\big )
\\[1ex]
=
\big (
e^{-i\scal {BD_\xi}{D_x}}\fkc _1 
\big )
\wpr _B
\big (
e^{-i\scal {BD_\xi}{D_x}}\fkc _2 
\big ),
\end{multline}
when $\fkc _1,\fkc_2 \in \Sigma _1'(\rr {2d})$, provided
the symbol products make sense as elements in
$\Sigma _1'(\rr {2d})$.

\par

Recall that $\ON =\ON (L^2(\rr d))$ is
the set of all orthonormal sequences
$\{ f_j\} _{j=1}^\infty$ in $L^2(\rr d)$.
In what follows, let
$\ON _0=\ON _0(L^2(\rr d))$ be the subset
of $\ON (L^2(\rr d))$, which consists of
all \emph{finite} orthonormal sequences,
$\{ f_j\} _{j=1}^N$ in $L^2(\rr d)$ such that
$f_j\in \Sigma _1(\rr d)$ for every
$j\in \mathbf Z_+$.
We also let $s_{0,A}(\rr {2d})$ be the set of all 
$\fka \in \Sigma _1(\rr {2d})$ such that
for some integer $N\ge 1$, $\{ f_j\} _{j=1}^\infty \in \ON _0$,
$\{ g_j\} _{j=1}^\infty \in \ON _0$
and decreasing non-negative sequence
$\{ \lambda _j\} _{j=1}^\infty$, we have
\begin{equation}\label{Eq:FiniteWignExp}
\fka =\sum _{j=1}^N \lambda _jW_{f_j,g_j}^A .
\end{equation}
It follows that $s_{0,A}(\rr {2d})$ is a vector space
contained in $\Sigma _1(\rr {2d})$.

\par

A combination of Proposition
\ref{Prop:WignerPseudoLinks}
and spectral theorem (Proposition
\ref{Prop:SpectralThm})
gives the first part of
the following. Here recall that
$\check \fka (X) = \fka (-X)$.

\par

\begin{prop}\label{Prop:SchattenSymb1}
Let $p\in \mathbf R_\sharp$,
$A\in \GL (d,\mathbf R)$ and
$\fka \in s_p^A(\rr {2d})$.
Then the following is true:
\begin{enumerate}
\item if $p<\infty$, then there are
$\{ f_j \} _{j=1}^\infty \in \ON (L^2(\rr d))$ and
$\{ g_j \} _{j=1}^\infty \in \ON (L^2(\rr d))$
and a unique sequence
$\{ \lambda _j\} _{j=1}^\infty \in \ell ^p(\mathbf Z_+)$
of decreasing non-negative real numbers such that
\begin{equation}\label{Eq:SymbSpectRes}
\fka =\sum _{j=1}^\infty \lambda _jW_{f_j,g_j}^A 
\quad \text{and}\quad
\nm \fka{s_{p,A}} = (2\pi )^{-\frac d2}
\nm {\{ \lambda _j\} _{j=1}^\infty}{\ell ^p}\text ;
\end{equation}

\vrum

\item if $p<\infty$, then $s_{0,A}(\rr {2d})$
is dense in $s_{p,A}(\rr {2d})$;

\vrum

\item if $X_0=(x _0,\xi _0)\in \rr {2d}$,
then the mappings
\begin{equation}
\label{Eq:SchattenSymbMappings}
\fka \mapsto \fka (\cdo -X_0)
\quad \text{and}\quad
\fka \mapsto \check \fka
\end{equation}
are bijective on $s_{0,A}(\rr {2d})$
and bijective isometries on $s_{p,A}(\rr 
{2d})$.
\end{enumerate}
\end{prop}

\par

For the proof it is convenient
to have the following lemma.
The result follows by
straight-forward computations
and is left for the reader (see also
\cite[Lemma 1.3]{Toft3}).

\par

\begin{lemma}
\label{Lemma:WignerCheckTranslations}
Let $A\in \GL (\mathbf R,d)$,
and $T_{X_0}$ be the map
on $\Sigma _1'(\rr d)$, given by
$$
T_{X_0}f
\equiv
f(\cdo -x_0)e^{-i\scal \cdo {\xi _0}},
\qquad f\in \Sigma _1'(\rr d),
$$
when $X_0=(x_0,\xi _0)\in \rr {2d}$.
Then
$$
W_{f,g}^A(X-X_0)
=
W_{T_{X_0}f,T_{X_0}g}^A(X)
\quad \text{and}\quad
W_{f,g}^A(-X)
=
W_{\check f,\check g}(X),
$$
when $f,g\in \Sigma _1'(\rr d)$.
\end{lemma}

\par

\begin{proof}[Proof of Proposition
\ref{Prop:SchattenSymb1}]
The assertion (1) follows from Proposition
\ref{Prop:SpectralThm} and the fact that
the map \eqref{Eq:SympPsDOMap}
is isometric bijection
between $s_{p,A}(\rr {2d})$ and 
$\mascI _p(L^2(\rr d))$.

\par

The assertion (2) follows from the same
arguments and the fact that the set of
finite rank operators are dense in
$\mascI _p(L^2(\rr d))$ when $p<\infty$.

\par

In order to prove (3) we first observe
that if $T_{X_0}$ is the same as in
Lemma
\ref{Lemma:WignerCheckTranslations},
then the map
$$
\{ f_j\} _{j=1}^\infty 
\mapsto
\{ T_{X_0}f_j\} _{j=1}^\infty 
$$
is bijective on $\ON _0(L^2(\rr d))$
and on $\ON (L^2(\rr d))$.
The assertion now follows by
combining this fact with (1) and
Lemma
\ref{Lemma:WignerCheckTranslations}.
%
%
%
\end{proof}

\par

\begin{rem}\label{Rem:SchattenWigner}
Suppose that $f,g\in L^2(\rr d)$,
$p\in \mathbf R_\sharp \cup \{ 0\}$
and $A\in \GL (d,\mathbf R)$. Then
it follows by Fourier's inversion
formula that
$$
\nm {W_{f,g}^A}{L^2(\rr {2d})}
=
\nm f{L^2(\rr d)}\nm g{L^2(\rr d)},
$$
which is also equivalent with the
so-called Moyal's formula.
Hence, by previous proposition
it follows that
$W_{f,g}^A\in s_{p,A}(\rr {2d})$, and
\begin{equation}\label{Eq:SchattenWigner}
\nm {W_{f,g}^A}{s_{p,A}}= (2\pi )^{-\frac d2}
\nm f{L^2(\rr d)}\nm g{L^2(\rr d)}.
\end{equation}
\end{rem}

\par

\begin{rem}
Let $A,B\in \GL (d,\mathbf R)$ and
$p\in \mathbf R_\sharp \cup \{ 0\}$. Then
it follows from
\eqref{Eq:CalculiTransfer} that
\begin{equation}
\label{Eq:SchattenSymbTransf}
e^{i\scal {AD_\xi}{D_x}}s_{p,A}(\rr {2d})
=
e^{i\scal {BD_\xi}{D_x}}s_{p,B}(\rr {2d}).
\end{equation}
If in addition $p>0$, then it also follows
that the map
$e^{i\scal {(A-B)D_\xi}{D_x}}$ is an isometric
bijection from $s_{p,A}(\rr {2d})$ to
$s_{p,B}(\rr {2d})$.
\end{rem}

\par

\begin{rem}
Let $p\in [1,\infty ]$, $r\in (0,1]$
and $A\in \GL (\mathbf R,d)$. Then  
Proposition \ref{Prop:SchattenSymb1}
shows that
\begin{align}
\WL ^{1,r}(\rr {2d})*s_{p,A}(\rr {2d})
&\subseteq
L^1(\rr {2d})*s_{p,A}(\rr {2d})
\notag
\\[1ex]
&\subseteq
\maclM (\rr {2d})*s_{p,A}(\rr {2d})
=
s_{p,A}(\rr {2d}).
\label{Eq:ConvL1Schatten}
\end{align}
Here recall that $\maclM (\rr {2d})$
is the Banach space of
(complex-valued) measures on $\rr {2d}$
with bounded mass.
(Cf. Example \ref{Ex:BanAlgUnital}.)

\par

In fact, the first two inclusions in
\eqref{Eq:ConvL1Schatten} are obvious.
By \eqref{Eq:SchattenSymbMappings}
it follows that if $\fka$ is a
simple function on $\rr {2d}$
and $\fkb \in s_{p,A}(\rr {2d})$,
then $\fka *\fkb \in s_{p,A}(\rr {2d})$,
and
\begin{equation}
\label{Eq:L1SchattenConv}
\nm {\fka *\fkb}{s_{p,A}}
\le
\nm {\fka}{L^1}\nm {\fkb}{s_{p,A}},
\quad \fka \in L^1(\rr {2d}),\
\fkb  \in s_{p,A}(\rr {2d}).
\end{equation}
For general $\fka \in L^1(\rr {2d})$,
the convolution
$\fka *\fkb \in s_{p,A}(\rr {2d})$
and estimate
\eqref{Eq:L1SchattenConv}
now follow by continuous extensions,
using the fact that the set of
simple functions are dense in
$L^1(\rr {2d})$.

\par

Since $\maclM (\rr {2d})$ is translation
invariant, similar arguments also
show that for
$\fka \in \maclM (\rr {2d})$
and $\fkb \in s_{p,A}(\rr {2d})$,
then $\fka *\fkb$ is uniquely
defined as an element in
$s_{p,A}(\rr {2d})$. That is,
$\maclM (\rr {2d})*s_{p,A}(\rr {2d})
\subseteq
s_{p,A}(\rr {2d})$ Since
$\maclM$ contains the unit element
$\delta _0$ for convolutions, it
follows that equality is attained
for the last inclusion, giving
\eqref{Eq:ConvL1Schatten}.

\par

We also observe that a
combination of \eqref{Eq:ConvL1Schatten},
\eqref{Eq:L1SchattenConv} and
Hahn-Banach's theorem give
\begin{equation}
\label{Eq:MeasSchattenConv}
\nm {\fka *\fkb}{s_{p,A}}
\le
\nm {\fka}{\maclM}\nm {\fkb}{s_{p,A}},
\quad \fka \in \maclM (\rr {2d}),\
\fkb  \in s_{p,A}(\rr {2d}).
\end{equation}
\end{rem}

\par

Next we shall discuss compositions and 
other multiplications
for Schatten-von Neumann symbols. The 
following result concerns H{\"o}lder 
properties for such symbol classes, and 
follows from Proposition
\ref{Prop:SchattenHolder}. The details are
left for the reader. 

\par

\begin{prop}\label{Prop:SchattenPseudoHolder1}
Let $p,q,r\in \mathbf R_\sharp$ satisfy
$$
\frac 1p+\frac 1q=\frac 1r
\quad \text{and}\quad
(p,q,r)\neq (\infty ,\infty ,\sharp ),
$$
$A\in \GL (d,\mathbf R)$. Then the map
$(\fka ,\fkb )\mapsto \fka \wpr _A\fkb$ from
$\Sigma _1(\rr {2d})\times \Sigma _1(\rr {2d})$
to $\Sigma _1(\rr {2d})$ is uniquely extendable 
to a continuous map from $s_{p,A}(\rr {2d})
\times s_{q,A}(\rr {2d})$ to $s_{r,A}(\rr {2d})$,
and
$$
\nm {\fka \wpr _A\fkb}{s_{r,A}}
\le
\nm {\fka}{s_{p,A}}\nm {\fkb}{s_{q,A}},
\qquad
\fka \in s_{p,A}(\rr {2d}),\ \fkb \in s_{q,A}(\rr {2d}).
$$
\end{prop}

\par

\begin{proof}
The result follows from
the fact that $\fka \mapsto \op _A(\fka )$ is an isometric
bijection from $s_{p,A}(\rr {2d})$ to $\mascI _p$,
and that
$$
T_1\circ T_2\in \mascI _r
\quad \text{and}\quad
\nm {T_1\circ T_2}{\mascI _r}
\le
\nm {T_1}{\mascI _p}\nm {T_2}{\mascI _q}
$$
when $T_1\in \mascI _p$ and $T_2\in \mascI _q$.
(See e.{\,}g. \cite{Sim}.)
\end{proof}


\par

\subsection{The Weyl calculus, symplectic
Fourier transform and twisted convolution}
\label{subsec5.3}

\par

Some properties in pseudo-differential
operators take convenient forms in
the case of Weyl calculus,
i.{\,}e. when $A=\frac 12 \cdot I_d$.
For convenience we set
\begin{equation}\label{Eq:WeylNot}
\begin{alignedat}{2}
\op ^w(\fka ) &= \op _A(\fka ), &
\quad
s_{p,A} &= s_p^w,
\quad
W_{f,g}=W_{f,g}^A
\\[1ex]
K_{\fka}^w &= K_{\fka ,A} &
\quad \text{and}\qquad 
\fka \wpr \fkb &= \fka \wpr _A \fkb ,
\quad \text{when}\quad
A={\textstyle{\frac 12}} \cdot I_d,
\end{alignedat}
\end{equation}
when $f,g\in \Sigma _1'(\rr d)$ and
$\fka ,\fkb \in \Sigma _1'(\rr {2d})$
are suitable.

\par

The Weyl quantization possess several
properties which are violated for other
pseudo-differential calculi. For example,
it is the only pseudo-differential
calculus which is so-called
symplectic invariant, a fundamental
property in quantization.
%
%

\par

There are several convenient links between
the Weyl quantization, symplectic Fourier
transform and the twisted convolution
(see \eqref{Eq:SympFourTrans}
and \eqref{Eq:TwistConv}).
For example, beside \eqref{Eq:FourTwist}
we have
\begin{equation}
\label{Eq:WeylTwistConvRel0}
\begin{aligned}
\fka \wpr \fkb
&= 
(2\pi )^{-\frac d2}\fka
*_\sigma (\mascF _\sigma \fkb) =
(2\pi )^{-\frac d2}(\mascF _\sigma \fka )
*_\sigma \check \fkb ,
\\[1ex]
\mascF _\sigma (\fka \wpr \fkb )
&=
(2\pi )^{-\frac d2}(\mascF _\sigma \fka )
*_\sigma (\mascF _\sigma \fkb ),
\end{aligned}
\end{equation}
\begin{alignat}{2}
K_{\mascF _\sigma \fka}^w
&=
S\circ K_{\fka}^w,&
\quad
(SK)(x,y) &= K(-x,y)
\label{Eq:WeylTwistConvRel1}
\intertext{and}
\mascF _\sigma W_{f,g}
&=
W_{\check f,g}, &
\quad
\check f(x) &= f(x),
\label{Eq:WeylTwistConvRel2}
\end{alignat}
when $\fka ,\fkb \in \maclS _s'(\rr {2d})$
are suitable and $f,g\in \maclS _s'(\rr d)$
(cf. \eqref{Eq:WeylNot}).
We also have that
\begin{equation}\label{Eq:SympFourSchattenCl}
\mascF _\sigma : s_p^w(\rr {2d})
\to
s_p^w(\rr {2d})
\end{equation}
is an isometric bijection for every
$p\in \mathbf R_\sharp \cup \{ 0\}$.

\par

A combination of Propositions
\ref{Prop:SchattenSymb1} (3) and
\ref{Prop:SchattenPseudoHolder1},
\eqref{Eq:WeylTwistConvRel0} and
\eqref{Eq:SympFourSchattenCl} gives
the following. The details are left for the reader.

\par

\begin{prop}
\label{Prop:SchattenTwistedHolder1}
Let $p$, $q$ and $r$ be as in Proposition
\ref{Prop:SchattenPseudoHolder1}.
Then the map $(\fka ,\fkb )\mapsto \fka *_\sigma \fkb$
is continuous from $s_{p}^w(\rr {2d})\times s_{q}^w(\rr {2d})$
to $s_{r}^w(\rr {2d})$, and
$$
\nm {\fka *_\sigma \fkb}{s_r^w}
\le
(2\pi )^{-\frac d2} \nm {\fka}{s_{p}^w}\nm {\fkb}{s_{q}^w},
\qquad
\fka \in s_{p}^w(\rr {2d}),\ \fkb \in s_{q}^w(\rr {2d}).
$$
\end{prop}

\par

For future references we also recall that
\begin{equation}
\label{Eq:L1GivesContOps}
\begin{aligned}
\WL ^{1,r}(\rr {2d})
\subseteq L^1(\rr {2d})
&\subseteq
s_\sharp ^w(\rr {2d})\cap \maclM (\rr {2d})
\\[1ex]
&\subseteq
s_\sharp ^w(\rr {2d})\cup \maclM (\rr {2d})
\subseteq s_\infty ^w(\rr {2d}),
\quad
r\in (0,1],
\end{aligned}
\end{equation}
with continuous inclusions.
(See e.{\,}g. \cite{Hor1}.)

\par

\subsection{Convolutions of
Schatten-von Neumann symbols}

\par

For convolutions we have the following.

\par

\begin{prop}\label{Prop:YoungIneqSchattenClasses}
Suppose $p,q,r\in (0,\infty ]$ and
$A\in \GL (d,\mathbf R)$ satisfy
$$
\frac 1p+\frac 1q = 1+\frac 1r,
\qquad p,q,r\ge 1,
$$
or $q\le p\le r\le 1$.
Then the map $(\fka ,\fkb )\mapsto \fka *\fkb$ from
$\Sigma _1(\rr {2d})\times \Sigma _1(\rr {2d})$ to
$\Sigma _1(\rr {2d})$ is uniquely extendable to
a continuous map from
$$
\WL ^{\max (q,1),q}(\rr {2d})
\times
s_{p,A}(\rr {2d})
\quad \text{or}\quad
s_{p,A}(\rr {2d})
\times
\WL ^{\max (q,1),q}(\rr {2d})
$$
to $s_{r,A}(\rr {2d})$.
\end{prop}

\par

\begin{proof}
By \eqref{Eq:SchattenSymbTransf} and that
\begin{equation}\label{Eq:ExpFourMultConv}
e^{i\scal {AD_\xi }{D_x}}(\fka *\fkb)
=
(e^{i\scal {AD_\xi }{D_x}}\fka )*\fkb,
\end{equation}
we reduce ourself to the Weyl case
$A=\frac 12 I_d$.
For $p,q,r\ge 1$, the result follows from
\cite{Toft3,Wer1}. 

\par

We need to consider the case when
$q\le p \le r\le 1$. Then the
facts that $s_{p,A}(\rr {2d})$ and
$\WL ^{1,q}(\rr {2d})$ increase with
$p$ and $q$, respectively, reduce
ourselves to the case when $p=q=r\le 1$.

\par

First we assume that $\fka = W_{f,f}^A=W_{f,f}$
with $f\in L^2(\rr d)$, and that
$0\le \fkb \in \WL ^{1,r}(\rr {2d})$.
Then $\fka$ is continuous and real-valued.
Let $Q=[0,1]^{2d}$.
The mean-value theorem gives that
for some $Z_j\in Q$, $j\in \zz d$, 
\begin{align*}
\nm {\fka *\fkb}{s_r^w}^r
&=
\NM {\sum _{j\in \zz d}\int _{j+Q}
\fka (\cdo -Y)\fkb (Y)\, dy}{s_r^w}^r
\\[1ex]
&=
\NM {\sum _{j\in \zz d}\fka (\cdo -j-Z_j)
\int _{j+Q} \fkb (Y)\, dy}{s_r^w}^r
\\[1ex]
&\le
\sum _{j\in \zz d}\nm {\fka (\cdo -j-Z_j)}{s_r^w}^r 
\left (\int _{j+Q} \fkb (Y)\, dy\right )^r
\\[1ex]
&=
\sum _{j\in \zz d}\nm {\fka}{s_r^w}^r 
\left (\int _{j+Q} \fkb (Y)\, dy\right )^r
=
\nm {\fka}{s_r^w}^r \nm {\fkb}{\WL ^{1,r}}^r,
\end{align*}
and the result follows in this case.

\par

Next assume that $\fka = W_{f,f}$ as before
but $\fkb \in \WL ^{1,r}(\rr {2d})$ being general.
By splitting up $\fkb$ in positive and
negative real and imaginary parts
$$
\fkb = \fkb _{1} - \fkb _{2} +i(\fkb _{3}-\fkb _{4}),
$$
with
\begin{alignat*}{2}
\fkb _{1} &= \max (\RE (\fkb ) ,0), &
\quad
\fkb _{2} &= \max (-\RE (\fkb ) ,0),
\\[1ex]
\fkb _{3} &= \max (\IM (\fkb ) ,0) &
\quad \text{and}\quad
\fkb _{4} &= \max (-\IM (\fkb ) ,0),
\end{alignat*}
the first part of the proof gives
\begin{align*}
\nm {\fka *\fkb}{s_r^w}^r
&\le
\nm {\fka *\fkb _{1}}{s_r^w}^r
+
\nm {\fka *\fkb _{2}}{s_r^w}^r
+
\nm {\fka *\fkb _{3}}{s_r^w}^r
+
\nm {\fka *\fkb _{4}}{s_r^w}^r
\\[1ex]
&\le
\nm {\fka}{s_r^w}^r
\sum _{k=1}^4
\nm {\fkb _{k}}{\WL ^{1,r}}^r
\le
4\nm {\fka}{s_r^w}^r \nm {\fkb}{\WL ^{1,r}}^r,
\end{align*}
and the result follows in this case as well.

\par

Next suppose that $\fka = W_{f,g}$ with
$\nm f{L^2}=\nm g{L^2}=1$
and that $\fkb \in \WL ^{1,r}(\rr {2d})$.
Then
$$
\fka
=
\frac 14(\fka _1-\fka _2
+
i(\fka _3-\fka _4)),
$$
where
\begin{alignat*}{2}
\fka _k
&=
W_{f+\theta _kg,f+\theta _kg}, &
\quad
\theta _1 =1,\ \theta _2 = -1,\ \theta _3 = i
\quad \text{and}\quad
\theta _4 = -i.
\end{alignat*}
We have
\begin{align*}
\nm {\fka _k}{s_r^w}
&=
\nm {W_{f+\theta _kg,f+\theta _kg}}{s_r^w}
=
(2\pi )^{\frac d2}\nm {f+\theta _kg}{L^2}^2
\\[1ex]
&\le
(2\pi )^{\frac d2}(\nm f{L^2}+\nm g{L^2})^2
=
4(2\pi )^{\frac d2}=4\nm {\fka}{s_r^w}.
\end{align*}

\par

By the previous case we get
\begin{align*}
\nm {\fka *\fkb}{s_r^w}^r
&\le
4^{-r} \sum _{k=1}^4 \nm {\fka _k*\fkb}{s_r^w}^r
\\[1ex]
&\le
4^{-r} \sum _{k=1}^4 4
\nm {\fka _k}{s_r^w}^r \nm {\fkb}{\WL ^{1,r}}^r
\\[1ex]
&\le
16 \nm {\fka}{s_r^w}^r
\nm {\fkb}{\WL ^{1,r}}^r,
\end{align*}
and the result follows in this case as well.

\par

Finally, for general $\fka$ given by 
\eqref{Eq:SymbSpectRes} with $p=r$,
we get
\begin{align*}
\nm {\fka *\fkb}{s_r^w}^r
&=
\NM {\sum _{j=1}^\infty
\lambda _jW_{f_j,g_j}*\fkb }{s_r^w}^r
\le
\sum _{j=1}^\infty \lambda _j^r
\nm {W_{f_j,g_j}*\fkb}{s_r^w}^r
\\[1ex]
&\lesssim
\left ( \sum _{j=1}^\infty \lambda _j^r \right )
\nm \fkb {\WL ^{1,r}}^r
\asymp
\nm {\fka}{s_r^w}^r\nm \fkb {\WL ^{1,r}}^r,
\end{align*}
and the result follows.
\end{proof}

\par

%


%

\par

\par

\subsection{Convolution factorizations of
Schatten-von Neumann symbols}

\par

The following result shows that any Schatten-von Neumann
symbol can be factorized through convolutions
into elements in the same class and elements in the
Wiener amalgam space $\WL ^{1,r}(\rr {2d})$,
for suitable $r\in (0,1]$.

\par

\begin{thm}\label{Thm:FactSchattenWigner}
Suppose $r\in (0,1]$, $p\in [r,\sharp ]$
and $A\in \GL (d,\mathbf R)$. Then 
\begin{equation}
\WL ^{1,r}(\rr {2d})*s_{p,A}(\rr {2d})
=
s_{p,A}(\rr {2d})* \WL ^{1,r}(\rr {2d})
=
s_{p,A}(\rr {2d}).
\end{equation}
\end{thm}

\par

\begin{proof}
By \eqref{Eq:SchattenSymbTransf} and
\eqref{Eq:ExpFourMultConv}, we reduce ourself to
the Weyl case $A=\frac 12I_d$. In particular
it follows that
$$
W_{f,f}(x,\xi )=W_{f,f}^A(x,\xi )\in \mathbf R
$$
when $f\in L^2(\rr d)$ and $x,\xi \in \rr d$.
By the commutativity of the convolution
product it suffices to prove
$$
s_{p,A}(\rr {2d})* \WL ^{1,r}(\rr {2d})
=
s_{p,A}(\rr {2d}).
$$

\par

Let $\phi \in C_0^\infty (\rr {2d};[0,1])$
be such that
$$
\int _{\rr {2d}}\phi (Y)\, dY=1
\quad \text{and}\quad
\supp (\phi) \subseteq Q_M\equiv [-1,1]^{2d},
$$
and let
$\phi _\ep (X)=\ep ^{-2d}\phi (\ep ^{-1}X)$
for any $\ep >0$, $X=(x,\xi )\in \rr {2d}$.
It follows by straight-forward computations
that if $f,g\in L^2(\rr d)$, then
$$
(W_{f,g}*\phi _\ep )(x,\xi )
=
\int _{\rr {2d}}
W_{f_{\ep Y},g_{\ep Y}}(x,\xi )\phi (Y)\, dY ,
$$
where
$$
f_Y (x)=f(x-y)e^{-i\scal x\eta},
\quad
g_Y (x)=g(x-y)e^{-i\scal x\eta},
\quad
Y=(y,\eta )\in \rr {2d}.
$$
Hence, if $\fka$ is expressed as in \eqref{Eq:SymbSpectRes}, then
\begin{align}
\fka *\phi _\ep -\fka
&=
\fkb _\ep +\fkc _\ep
\intertext{where}
\fkb _\ep &=
\int _{Q_M} \left (
\sum _{j=1}^\infty \lambda _jW_{f_{j,\ep Y},g_{j,\ep Y}-g_j}
\right )
\phi (Y)\, dY
\\[1ex]
&=
\sum _{j=1}^\infty \lambda _jW_{f_{j,\ep Y_0},g_{j,\ep Y_0}-g_j}
\intertext{and}
\fkc _\ep &=
\int _{\rr {2d}} \left (
\sum _{j=1}^\infty \lambda _jW_{f_{j,\ep Y}-f_j,g_j}
\right )
\phi (Y)\, dY .
\end{align}
Since $\nm {\phi _\ep}{\WL ^{1,r}}\le C$
for some constant $C>0$ which is
independent of $\ep \in (0,1]$,
the assertion will follow if we prove
\begin{equation}\label{Eq:LimitIdentSchatten}
\lim _{\ep \to 0+}\nm {\fkb _\ep}{s_{p}^w} = 0
\quad \text{and}\quad
\lim _{\ep \to 0+}\nm {\fkc _\ep}{s_{p}^w} = 0,
\end{equation}
in view of Theorem \ref{Thm:ApprIdent}.

\par

We have
$$
\fka -\sum _{j=1}^N \lambda _jW_{f_j,g_j} \to 0
$$
as $N\to \infty$ with convergence in $s_{p,A}(\rr {2d})$. By combining
this with Proposition \ref{Prop:YoungIneqSchattenClasses},
it follows that
$$
\left (\fka -\sum _{j=1}^N \lambda _jW_{f_j,g_j}\right )*\phi \to 0
$$
as $N\to \infty$, again with convergence in $s_{p}^w(\rr {2d})$.
From these limits we reduce ourself
to the case where we may assume that
\begin{equation}\label{Eq:FiniteRankSymb}
\fka =\sum _{j=1}^N \lambda _jW_{f_j,g_j} ,
\end{equation}
for some integer $N\ge 0$. 

\par

Our aim is to apply the mean-value theorem and Minkowski's
inequality to reach desired estimates. We then need to
reformulate $\fkb _\ep$ in suitable ways. We have
\begin{align}
\fkb _\ep
&=
\frac 14
\sum _{j=1}^N
\lambda _j\fkb _{j,\ep},
\qquad
\fkb _{j,\ep} = \fkb _{1,j,\ep}+i\fkb _{2,j,\ep}-\fkb _{3,j,\ep}-i\fkb _{4,j,\ep},
\label{Eq:bepFirstDecomp}
\intertext{where}
\fkb _{k,j,\ep}
&=
\int _{Q_M}W_{h_{k,j,\ep Y},h_{k,j,\ep Y}}\phi (Y)\, dY,
\qquad
h_{k,j,Y} = f_{j,Y}+i^{k-1}(g_{j,Y}-g),
\end{align}
$k=1,2,3,4$. Since $\phi \ge 0$ satisfies
$$
\int _{\rr {2d}}\phi (Y)\, dY =1
$$
and $W_{h_{k,j,\ep Y},h_{k,j,\ep Y}}$
is real-valued and smooth, the mean value theorem gives
$$
\fkb _{k,j,\ep}
=
W_{h_{k,j,\ep Y_{k,j}},h_{k,j,\ep Y_{k,j}}},
$$
for some $Y_{k,j}=Y_{k,j,\ep}\in Q_M$.

\par

By rearranging terms we obtain
\begin{align}
\fkb _{j,\ep}
&=
\fkb ^0_{1,j,\ep} + i\fkb ^0_{2,j,\ep} + \fkb ^0_{3,j,\ep} + i\fkb ^0_{4,j,\ep}
\intertext{where}
\fkb ^0_{1,j,\ep}
&=
4W_{f_{j,\ep Y_{1,j}} , g_{j,\ep Y_{1,j}} -g_j},
\label{Eq:Decompb1jep}
\\[1ex]
\fkb ^0_{2,j,\ep}
&=
-W_{f_{j,\ep Y_{1,j}}-f_{j,\ep Y_{2,j}} , f_{j,\ep Y_{1,j}} +ig_{j,\ep Y_{1,j}}}
-i
W_{g_{j,\ep Y_{1,j}}-g_{j,\ep Y_{2,j}} , f_{j,\ep Y_{1,j}} +ig_{j,\ep Y_{1,j}}}
\notag
\\
&+i
W_{f_{j,\ep Y_{2,j}}+ig_{j,\ep Y_{2,j}} , f_{j,\ep Y_{1,j}} -f_{j,\ep Y_{2,j}}}
+
W_{f_{j,\ep Y_{2,j}}+ig_{j,\ep Y_{2,j}} , g_{j,\ep Y_{1,j}} -g_{j,\ep Y_{2,j}}}
\notag
\\
&+
2\operatorname{Im} \left (
W_{f_{j,\ep Y_{1,j}}-f_{j,\ep Y_{2,j}} , g_j}
-
iW_{g_{j,\ep Y_{1,j}}-g_{j,\ep Y_{2,j}} , g_j}
\right )
\label{Eq:Decompb2jep}
\\[1ex]
\fkb ^0_{3,j,\ep}
&=
W_{f_{j,\ep Y_{1,j}}-f_{j,\ep Y_{3,j}} , f_{j,\ep Y_{1,j}} -g_{j,\ep Y_{1,j}}}
-
W_{g_{j,\ep Y_{1,j}}-g_{j,\ep Y_{3,j}} , f_{j,\ep Y_{1,j}} -g_{j,\ep Y_{1,j}}}
\notag
\\
&+
W_{f_{j,\ep Y_{3,j}}-g_{j,\ep Y_{3,j}} , f_{j,\ep Y_{1,j}} -f_{j,\ep Y_{3,j}}}
-
W_{f_{j,\ep Y_{3,j}}-g_{j,\ep Y_{3,j}} , g_{j,\ep Y_{1,j}} -g_{j,\ep Y_{3,j}}}
\notag
\\
&+
2\operatorname{Re} \left (
W_{f_{j,\ep Y_{1,j}}-f_{j,\ep Y_{3,j}} , g_j}
-
W_{g_{j,\ep Y_{1,j}}-g_{j,\ep Y_{3,j}} , g_j}
\right )
\label{Eq:Decompb3jep}
\intertext{and}
\fkb ^0_{4,j,\ep}
&=
W_{f_{j,\ep Y_{1,j}}-f_{j,\ep Y_{4,j}} , f_{j,\ep Y_{1,j}} -ig_{j,\ep Y_{1,j}}}
-
iW_{g_{j,\ep Y_{1,j}}-g_{j,\ep Y_{4,j}} , f_{j,\ep Y_{1,j}} -ig_{j,\ep Y_{1,j}}}
\notag
\\
&+
W_{f_{j,\ep Y_{4,j}}-ig_{j,\ep Y_{4,j}} , f_{j,\ep Y_{1,j}} -f_{j,\ep Y_{4,j}}}
+
iW_{f_{j,\ep Y_{4,j}}-ig_{j,\ep Y_{4,j}} , g_{j,\ep Y_{1,j}} -g_{j,\ep Y_{4,j}}}
\notag
\\
&+
2\operatorname{Im} \left (
W_{f_{j,\ep Y_{1,j}}-f_{j,\ep Y_{4,j}} , g_j}
-
iW_{g_{j,\ep Y_{1,j}}-g_{j,\ep Y_{4,j}} , g_j}
\right ).
\label{Eq:Decompb4jep}
\end{align}

\par

Since $\overline {W_{f,g}(x,\xi )}=W_{g,f}(x,\xi )$, giving that
$$
2\operatorname{Re} ( W_{f, g})
=W_{f,g}+W_{g,f}
\quad \text{and}\quad
2\operatorname{Im} ( W_{f, g})
=i^{-1}(W_{f,g}-W_{g,f}),
$$
it follows that all terms in \eqref{Eq:Decompb1jep}--\eqref{Eq:Decompb4jep}
are of the forms
\begin{equation}\label{Eq:DescrWignerDistr}
\begin{alignedat}{3}
&W_{h_{1,j}, g_{j,\ep Y_{k,j}} -g_j}, &
\quad
&W_{f_{j,\ep Y_{k,j}}-f_{j,\ep Y_{m,j}} , h_{2,j}}, &
\quad
&W_{g_{j,\ep Y_{k,j}}-g_{j,\ep Y_{m,j}} , h_{3,j}},
\\[1ex]
&W_{h_{4,j},f_{j,\ep Y_{k,j}}-f_{j,\ep Y_{m,j}}} &
\quad & \phantom k \qquad \text{and} & \quad
&W_{h_{5,j},g_{j,\ep Y_{k,j}}-g_{j,\ep Y_{m,j}}},
\end{alignedat}
\end{equation}
for some $h_{k,j}\in \Sigma _1(\rr d)$, $k=1,\dots ,5$. 
Since $\Sigma _1(\rr d)$ is dense in $L^2(\rr d)$,
a straight-forward application of Lebesgue's theorem
gives 
$$
\lim _{\ep \to 0+}\nm {f_{j,\ep Y_{k,j}}-f_j}{L^2}
=
\lim _{\ep \to 0+}\nm {g_{j,\ep Y_{k,j}}-g_j}{L^2}
=0,
$$
which in turn leads to
$$
\lim _{\ep \to 0+}\nm {f_{j,\ep Y_{k,j}}-f_{j,\ep Y_{m,j}}}{L^2}
=
\lim _{\ep \to 0+}\nm {g_{j,\ep Y_{k,j}}-g_{j,\ep Y_{m,j}}}{L^2}
=0.
$$
An application of \eqref{Eq:SchattenWigner} now gives
\begin{multline*}
\lim _{\ep \to 0+}
\nm {W_{f_{j,\ep Y_{k,j}}-f_{j,\ep Y_{m,j}} , h_{2,j}}}{s_r^w}
\\[1ex]
=
\lim _{\ep \to 0+}
\left (
(2\pi )^{-\frac d2}
\nm {f_{j,\ep Y_{k,j}}-f_{j,\ep Y_{m,j}}}{L^2} 
\nm {h_{2,j}}{L^2}
\right )
=0
\end{multline*}
for the second expression in \eqref{Eq:DescrWignerDistr}.
In the same way we obtain
\begin{align*}
\lim _{\ep \to 0+}
\nm {W_{h_{1,j}, g_{j,\ep Y_{k,j}} -g_j}}{s_r^w}
&=
\lim _{\ep \to 0+}
\nm {W_{g_{j,\ep Y_{k,j}}-g_{j,\ep Y_{m,j}} , h_{3,j}}}{s_r^w}
\\[1ex]
&=
\lim _{\ep \to 0+}
\nm {W_{h_{4,j},f_{j,\ep Y_{k,j}}-f_{j,\ep Y_{m,j}}}}{s_r^w}
\\[1ex]
&=
\lim _{\ep \to 0+}
\nm {W_{h_{5,j},g_{j,\ep Y_{k,j}}-g_{j,\ep Y_{m,j}}}}{s_r^w}
=0
\end{align*}
for the other expressions in \eqref{Eq:DescrWignerDistr}.

\par

Since $\fkb ^0_{k,j,\ep}$ are linear combinations of
Wigner distributions in \eqref{Eq:DescrWignerDistr},
it now follows from the previous estimates and
the quasi-norm properties for $\nm \cdo{s_r^w}$ that
$$
\lim _{\ep \to 0+}\nm {\fkb ^0_{k,j,\ep}}{s_r^w} = 0,
\quad k=1,2,3,4,
$$
which leads to
$$
0\le \lim _{\ep \to 0+}\nm {\fkb ^0_{j,\ep}}{s_p^w} \le
\lim _{\ep \to 0+}\nm {\fkb ^0_{j,\ep}}{s_r^w} = 0.
$$
The first limit in \eqref{Eq:LimitIdentSchatten} now follows
from the previous limit and \eqref{Eq:bepFirstDecomp}.

\par

By similar arguments, the second limit in
\eqref{Eq:LimitIdentSchatten} follows. The details are
left for the reader, and the result follows.
\end{proof}

\par

\subsection{Symbol product and twisted convolution
factorizations of Schatten-von Neumann symbols}

\par

The following result shows among others that
any Schatten-von Neumann symbol in the
Weyl calculus can be factorized through twisted convolutions
or Weyl products into elements in the same class and
elements in the Wiener amalgam space
$\WL ^{1,r}(\rr {2d})$,
for suitable $r\in (0,1]$.

\par

\begin{thm}\label{Thm:SchattenTwistEqu}
Suppose $p\in (0,\sharp ]$ and $r\in (0,1]$.
Then
\begin{align}
s_{p}^w(\rr {2d})
&*_\sigma 
\WL ^{1,r}(\rr {2d})
=
s_{p}^w(\rr {2d})*_\sigma L^1(\rr {2d})
\notag
\\[1ex]
&=
s_{p}^w(\rr {2d})*_\sigma s_\sharp ^w(\rr {2d})
=
s_{p}^w(\rr {2d})*_\sigma s_\infty ^w(\rr {2d})
=
s_{p}^w(\rr {2d})
\label{Eq:SchattenTwistEqu1}
\intertext{and}
s_{p}^w(\rr {2d})
&\wpr
\WL ^{1,r}(\rr {2d})
=
s_{p}^w(\rr {2d})\wpr L^1(\rr {2d})
\notag
\\[1ex]
&=
s_{p}^w(\rr {2d})\wpr s_\sharp ^w(\rr {2d})
=
s_{p}^w(\rr {2d})\wpr s_\infty ^w(\rr {2d})
=
s_{p}^w(\rr {2d}).
\label{Eq:SchattenTwistEqu2}
\end{align}
\end{thm}

\par

\begin{proof}
By \eqref{Eq:WeylTwistConvRel0} and
\eqref{Eq:SympFourSchattenCl}, it follows that
\eqref{Eq:SchattenTwistEqu2} is equivalent to
\eqref{Eq:SchattenTwistEqu1}. Hence it suffices to
prove \eqref{Eq:SchattenTwistEqu1}. By
Proposition \ref{Prop:SchattenTwistedHolder1}
it follows that
$$
s_{p}^w(\rr {2d})*_\sigma s_\infty ^w(\rr {2d})
\subseteq
s_{p}^w(\rr {2d}).
$$
Hence it suffices to prove
\begin{equation}
\label{Eq:SchattenTwistEqu1Short}
s_{p}^w(\rr {2d})
*_\sigma
\WL ^{1,r}(\rr {2d})
=
s_{p}^w(\rr {2d})
\end{equation}
in view of \eqref{Eq:L1GivesContOps}.

\par

Let
\begin{align}
\phi (x,\xi )
&=
\left (
\frac \pi 2
\right )^{\frac d2}\psi (x)\psi (\xi ),
\notag
\intertext{where}
\psi &\in \Sigma _1(\rr d),\quad
\int _{\rr d}\psi (x)\, dx=1
\label{Eq:AlmOneIntegral}
\end{align}
and set $\phi _\ep
=
\ep ^{-2d}\phi (\ep ^{-1}\cdo )$.
Since $\nm {\phi _\ep}{\WL ^{1,r}}\le C$
for some constant $C>0$ which is
independent of $\ep \in (0,1]$,
the result follows from Theorem 
\ref{Thm:ApprIdent} if we prove that
\begin{equation}\label{Eq:TwistedConvAppr}
\lim _{\ep \to 0+}\nm {\phi _\ep *_\sigma \fka -\fka}{s_p^w}
=0,
\end{equation}
for every $\fka \in s_p^w(\rr {2d})$.

\par

By Proposition \ref{Prop:SchattenPseudoHolder1}
and \eqref{Eq:L1GivesContOps} it follows that
the map $(\fka ,\fkb )\mapsto \fka *_\sigma \fkb$
is continuous from $L^1(\rr {2d})\times s_p^w(\rr {2d})$
to $s_p^w(\rr {2d})$. Hence it suffices to prove
\eqref{Eq:TwistedConvAppr} when $\fka \in s_0^w(\rr {2d})$,
in view of Proposition \ref{Prop:SchattenSymb1}.
That is, we may assume that $\fka$ is given
by \eqref{Eq:FiniteRankSymb} for some
integer $N\ge 1$,
$\{ f_j \} _{j=1}^\infty \in \ON _0$,
$\{ g_j \} _{j=1}^\infty \in \ON _0$
and $\lambda _j\ge \lambda _{j+1}\ge 0$,
$j=1,\dots ,N-1$.

\par

Let $r=\min (1,p)$. By triangle inequality we get
\begin{align}
\nm {\phi _\ep *_\sigma \fka -\fka}{s_p^w}^r
\notag
&=
\NM {\sum _{j=1}^N \lambda _j
\big (
\phi _\ep *_\sigma W_{f_j,g_j} - W_{f_j,g_j}
\big )}{s_p^w}^r
\notag
\\[1ex]
&\le
\sum _{j=1}^N \lambda _j^r
\nm {\phi _\ep *_\sigma W_{f_j,g_j} - W_{f_j,g_j}}{s_p^w}^r
\notag
\\[1ex]
&\le
\sum _{j=1}^N \lambda _j^r
\nm {\phi _\ep *_\sigma W_{f_j,g_j} - W_{f_j,g_j}}{s_r^w}^r.
\label{Eq:ApprSchattenEst1}
\end{align}
Since
\begin{align*}
\phi _\ep *_\sigma W_{f_j,g_j}
&=
(2\pi )^{\frac d2}\phi _\ep
\wpr (\mascF _\sigma W_{f_j,g_j})
\\[1ex]
&=
(2\pi )^{\frac d2}\phi _\ep
\wpr W_{\check f_j,g_j}
=
(2\pi )^{\frac d2}
W_{\op ^w(\phi _\ep)\check f_j,g_j},
\end{align*}
in view of Proposition
\ref{Prop:WignerPseudoLinks} and
\eqref{Eq:WeylTwistConvRel0}, we get
\begin{equation}\label{Eq:ApprSchattenEst2}
\begin{aligned}
\nm {\phi _\ep *_\sigma W_{f_j,g_j} - W_{f_j,g_j}}{s_r^w}
&=
\nm {W_{h_{\ep ,j},g_j}}{s_r^w}
\\[1ex]
&=
(2\pi )^{-\frac d2}\nm {h_{\ep ,j}}{L^2}\nm {g_j}{L^2}
=
(2\pi )^{-\frac d2}\nm {h_{\ep ,j}}{L^2},
\end{aligned}
\end{equation}
where
$$
h_{\ep ,j} = (2\pi )^{\frac d2}
\op ^w(\phi _\ep)\check f_j - f_j.
$$
Hence, in view of \eqref{Eq:ApprSchattenEst1}
and \eqref{Eq:ApprSchattenEst2}, the result will
follow if we prove
\begin{equation}\label{Eq:SearchLimTwistedAlg}
\lim _{\ep \to 0+}\nm {h_{\ep ,j}}{L^2}=0.
\end{equation}

\par

By \eqref{Eq:AlmOneIntegral} we have
\begin{align*}
h_{\ep ,j}(x)
&=
2^{-d}\ep ^{-2d}
\iint _{\rr {2d}}
\psi ((2\ep )^{-1}(x+y))\psi (\ep ^{-1}\xi )
f_j(-y)e^{i\scal {x-y}\xi}\, dyd\xi -f_j(x)
\\[1ex]
&=
2^{-d}\ep ^{-2d}
\iint _{\rr {2d}}
\psi ((2\ep )^{-1}y)\psi (\ep ^{-1}\xi )
f_j(x-y)e^{i\scal {x-y}\xi}\, dyd\xi -f_j(x)
\\[1ex]
&=
\iint _{\rr {2d}}
\psi (y)\psi (\xi )
f_j(x-2\ep y)e^{i\scal {x-2\ep y}\xi}\, dyd\xi -f_j(x)
\\[1ex]
&=
\iint _{\rr {2d}}
\psi (y)\psi (\xi )
f_j(x-2\ep y)e^{i\scal {x-2\ep y}\xi}\, dyd\xi -f_j(x)
\\[1ex]
&=
\int _{\rr {d}}
\psi (y)(
(2\pi )^{\frac d2}\widehat \psi (\ep x-2\ep ^2y)
f_j(x-2\ep y)-f_j(x))\, dy ,
\end{align*}
and Minkowski's inequality gives
\begin{equation}\label{Eq:hepEst1}
\nm {h_{\ep ,j}}{L^2}
\le
\int _{\rr {d}}
|\psi (y)|\cdot \nm 
{(2\pi )^{\frac d2}\widehat
\psi (\ep \cdo -2\ep ^2y)
f_j(\cdo -2\ep y)-f_j}{L^2}\, dy.
\end{equation}
We have $(2\pi )^{\frac d2}\widehat \psi (0)=1$
in view of \eqref{Eq:AlmOneIntegral},
which implies
\begin{equation}\label{Eq:hepEst2}
\lim _{\ep \to 0+}
\nm {(2\pi )^{\frac d2}\widehat
\psi (\ep \cdo -2\ep ^2y)
f_j(\cdo -2\ep y)-f_j}{L^2}=0.
\end{equation}
In fact, the latter limit is evident when
$f_j\in C_0^\infty (\rr d)$, and follows
for $f_j\in L^2(\rr d)$ by density arguments,
using that $C_0^\infty (\rr d)$ is dense
in $L^2(\rr d)$.

\par

Furthermore, since
$$
(2\pi )^{\frac d2}\nm {\widehat \psi}{L^\infty}
\le
\nm \psi{L^1}
$$
we obtain
\begin{multline}\label{Eq:hepEst3}
y\mapsto |\psi (y)| \cdot \nm 
{(2\pi )^{\frac d2}\widehat
\psi (\ep \cdo -2\ep ^2y)
f_j(\cdo -2\ep y)-f_j}{L^2}
\\[1ex]
\le
|\psi (y)| (\nm \psi{L^1} +1)\nm {f_j}{L^2}
\in L^1(\rr d).
\end{multline}
A combination of
\eqref{Eq:hepEst1}--\eqref{Eq:hepEst3}
and Lebesgue's theorem now gives
\eqref{Eq:SearchLimTwistedAlg}, and the
result follows.
\end{proof}

\par

A combination of \eqref{Eq:SymbProdTransf}
and Theorem \ref{Thm:SchattenTwistEqu} gives
the following. The details are left for the reader.

\par

\begin{thm}\label{Thm:SchattenPseudoProdEqu}
Suppose $p\in (0,\sharp ]$ and $A\in \GL (d,\mathbf R)$.
Then 
\begin{align}
s_{p,A}(\rr {2d})\wpr _As_{\sharp ,A}(\rr {2d})
=
s_{p,A}(\rr {2d})\wpr _As_{\infty ,A}(\rr {2d})
=
s_{p,A}(\rr {2d}).
\label{Eq:SchattenPseudoEqu2}
\end{align}
\end{thm}

\par

\begin{example}
Let $r\in (0,1]$. Due to 
\eqref{Eq:L1GivesContOps},
a natural question arise whether
$\WL ^{1,r}(\rr {2d})$ is contained
in $s_p^w(\rr {2d})$ or not,
for some choice of $p\in (0,\infty )$.
Observe that for such $p$ we have
$s_p^w(\rr {2d})\subsetneq
s_\sharp ^w(\rr {2d})$. Our results now
give negative answer on this. That is,
we have
\begin{equation}\label{Eq:WienerNotInSchatt}
\WL ^{1,r}(\rr {2d})
\nsubseteq
\bigcup _{p\in \mathbf R_+}
s_p^w(\rr {2d})
\qquad
\Big (
\text{but}\quad
\WL ^{1,r}(\rr {2d})
\subseteq
s_\sharp ^w(\rr {2d})
\Big ).
\end{equation}

\par

In fact, let
$$
\mascA = s_\sharp ^w(\rr {2d}),
\quad
\mascA _0 = \WL ^{1,r}(\rr {2d})
\quad \text{and}\quad
\mascA _t = s_{1/t} ^w(\rr {2d}),
\ t>0,
$$
which are algebras under the twisted
convolution. Now recall that
$\mascA _0 = \WL ^{1,r}(\rr {2d})$
is a factorization algebra under
the twisted convolution
and that H{\"o}lder's inequality
holds for the $s_p^w$ classes
in view of Theorem \ref{Thm:SurjTwistConvWien} 
and Proposition
\ref{Prop:SchattenTwistedHolder1}.
Consequently, the hypothesis of
Proposition \ref{Prop:AlgCont} is fulfilled
for $G=(\mathbf R_+,+)$ with the usual
ordering. Since
$$
\WL ^{1,r}(\rr {2d})
\nsubseteq
L^2(\rr {2d})
=
s_2 ^w(\rr {2d})
=
\mascA _{1/2},
$$
it now follows from
Proposition \ref{Prop:AlgCont}
that \eqref{Eq:WienerNotInSchatt} holds.
\end{example}

\par

\appendix

\par

\section{Factorization properties for
quasi-Banach modules}\label{App:A}

\par

In this section we extend well-known 
factorization properties
for Banach algebras and modules
to quasi-Banach algebras and modules.
Especially we follow the framework in \cite[Section 2]{Hew}
and extend (2.5) Theorem in \cite{Hew} in such a way that
we assume that the involved spaces are quasi-Banach
algebras and modules, instead of the more restrictive
Banach algebras and modules.

\par

In the first part we deduce some preparatory results which are
needed for the proof of the main result Theorem
\ref{Thm:ApprIdent}, which appear in the last part of the section.

\par

We have the following lemma, similar to (2.3) Lemma in
\cite{Hew}. Here in what follows we let
$$
B_r = B_{r,\maclB}=B_r(0) = B_{r,\maclB}(0)
\equiv
\sets {\phi \in \maclB}{\nm \phi{\maclB}\le r}
$$
be the (closed) ball in $\maclB$, centered at $0$
and with radius $r$.

\par

\begin{lemma}\label{Lem:InvOpProp}
Let $\maclB$ be a quasi-Banach algebra of 
order $p\in (0,1]$,
$\maclM$ be a left quasi-Banach $\maclB$ 
module of order $p$, let
$r\in [1,\infty )$ and let
$$
T : B_{r,\maclB} \to \maclB _E
$$
be given by
\begin{alignat}{2}
T(\phi )
&\equiv
\left ( \frac {2r+1}{2r} \right ) 
\left (
\varrho
+
\sum _{k=1}^{\infty}(-1)^k(2r)^{-k}\phi ^k
\right ),&
\quad
\phi &\in B_{r,\maclB}.
\intertext{Then}
T(\phi )
&=
\left (
\frac {2r}{2r+1}\varrho +\frac 1{2r+1}\phi 
\right )^{-1},&
\quad
\phi &\in B_{r,\maclB},
\label{Eq:InvOpProp1}
\\[1ex]
\nm {T(\phi )\cdot f - f}{\maclM}
&\le
(1-2^{-p})^{-\frac 2p}\, \frac {2r+1}{4r}
\nm {\phi \cdot f-f}{\maclM},&
\quad \phi &\in B_{r,\maclB},\ f\in \maclM
\label{Eq:InvOpProp2}
\intertext{and}
\frac 1{(2^p+1)^{\frac 1p}}
\Big ( 2\, +\, &\frac 1r\Big )
\le
\nm {T(\phi )}{\maclB _E}
\le \frac 1{(2^p-1)^{\frac 1p}}
\Big ( 2+\frac 1r\Big ), &
\quad \phi &\in B_{r,\maclB}.
\label{Eq:InvOpProp3}
\end{alignat}
\end{lemma}

\par

\begin{proof}
The assertion \eqref{Eq:InvOpProp1} follows
by straight-forward computations
(see e.{\,}g. the proof of (i) in (2.3)
Lemma in \cite{Hew}).

\par

In order to prove \eqref{Eq:InvOpProp2} we observe that
$$
T(\phi )\cdot f-f = \left ( \frac {2r+1}{2r} \right ) 
\sum _{k=1}^{\infty}(-1)^k(2r)^{-k}(\phi ^k\cdot f-f).
$$
By applying quasi-norms on the latter identity we obtain
\begin{equation}\label{Eq:EstTphiDiff}
\nm {T(\phi )\cdot f-f}{\maclM}^p
\le
\left (
\frac {2r+1}{2r}
\right )^p
\sum _{k=1}^\infty (2r)^{-kp}\nm {\phi ^k\cdot f-f}{\maclM}^p.
\end{equation}
For the norms in the sum we have
\begin{align*}
\nm {\phi ^k\cdot f-f}{\maclM}^p
&\le
\sum _{j=1}^k\nm {\phi ^j\cdot f-\phi ^{j-1}\cdot f}{\maclM}^p
\\[1ex]
&\le
\sum _{j=1}^k\nm \phi{\maclB}^{(j-1)p}
\nm {\phi \cdot f-f}{\maclM}^p
\le
k r^{kp}\nm {\phi \cdot f-f}{\maclM}^p.
\end{align*}
By inserting this into \eqref{Eq:EstTphiDiff} we obtain
\begin{align*}
\nm {T(\phi )\cdot f-f}{\maclM}^p
&\le
\left (
\frac {2r+1}{2r}
\right )^p\left (
\sum _{k=1}^\infty (2r)^{-kp}kr^{kp}\right )\nm {\phi \cdot f-f}{\maclM}^p
\\[1ex]
&=
\left (
\frac {2r+1}{2r} \right )^p
\frac {2^{-p}}{(1-2^{-p})^2}
\nm {\phi \cdot f-f}{\maclM}^p,
\end{align*}
which gives \eqref{Eq:InvOpProp2}.

\par

It remains to prove \eqref{Eq:InvOpProp3}. Since
$$
\varrho = T(\phi )\left (
\frac {2r}{2r+1}\varrho +\frac 1{2r+1}\phi 
\right ),
$$
we obtain
\begin{align*}
1
&=
\nm \varrho{\maclB _E}^p
\le
\nm{T(\phi )}{\maclB _E}^p
\left (
\left ( \frac {2r}{2r+1}\right )^p
\nm \varrho{\maclB _E}^p
+\left (\frac 1{2r+1}\right )^p\nm \phi{\maclB _E}^p 
\right )
\\[1ex]
&\le
\nm{T(\phi )}{\maclB _E}^p
\left (
\left ( \frac {2r}{2r+1}\right )^p
+\left (\frac 1{2r+1}\right )^pr^p 
\right ),
\end{align*}
which is the same as the left inequality in \eqref{Eq:InvOpProp3}.

\par

In similar ways we have
\begin{align*}
\nm {T(\phi )}{\maclB _E}^p
&=
\left ( \frac {2r+1}{2r} \right )^p
\left (
\nm \varrho{\maclB _E}^p +\sum _{k=1}^{\infty}(2r)^{-pk}
\nm \phi{\maclB _E}^{pk}
\right )
\\[1ex]
&\le
\left ( \frac {2r+1}{2r} \right )^p
\left (
1 +\sum _{k=1}^{\infty}(2r)^{-pk}
r^{pk}
\right )
=
\left ( \frac {2r+1}{2r} \right )^p
\left (
1 +\frac {2^{-p}}{1-2^{-p}}
\right ),
\end{align*}
which is equivalent to the right inequality in \eqref{Eq:InvOpProp3},
and the result follows.
\end{proof}

\par

The next result corresponds to (2.4) Lemma in \cite{Hew}.

\par

\begin{lemma}\label{Lem:RefInvProp}
Let $\maclB$ be a quasi-Banach algebra of order $p$,
$\maclM$ be a left quasi-Banach $\maclB$ module of
order $p\in (0,1]$, $\ep >0$ and $f\in \maclM$.
Also assume that $\maclB$ possess a bounded left approximate
identity for $(\maclB ,\maclM)$
of order $r\in [1,\infty )$.
Then there is a sequence $\{\phi _n\} _{n=1}^\infty \subseteq B_{r,\maclB}$
such that if
\begin{equation}\label{Eq:RefInvProp1}
\begin{aligned}
\psi _n
&=
\frac {(2r)^{n}}{(2r+1)^n}\varrho
+
\left (\sum _{k=1}^n
\frac {(2r)^{k-1}}{(2r+1)^k}\phi _k
\right )
\in \maclB _E,
\quad n\in \mathbf N,
\end{aligned}
\end{equation}
then $\psi _n^{-1}$ exists in $\maclB _E$, and
\begin{equation}\label{Eq:RefInvProp2}
\nm {\psi _n^{-1}\cdot f-\psi _{n-1}^{-1}\cdot f}{\maclM}\le 2^{-n}\ep ,
\quad
n\in \mathbf Z_+.
\end{equation}
\end{lemma}

\par

Here we interpret the right-hand side of \eqref{Eq:RefInvProp1}
as $\varrho$ when $n=0$. That is,
$$
\psi _0=\varrho .
$$

\par

\begin{proof}[Proof of Theorem \ref{Thm:ApprIdent}]
By Definition \ref{Def:ApprLeftUnit} and
\eqref{Eq:DefApprLeftUnitMod}, there is an 
element $\phi _1\in \maclB$
such that ${\phi _1}\in B_{r,\maclB}$ and
$$
\nm {\phi _1\cdot f-f}{\maclM}\le (1-2^{-p})^{\frac 2p}\, \frac {2r}{2r+1}\ep .
$$
Then
$$
\psi _1 = \frac {2r}{2r+1}\varrho
+
\frac 1{2r+1}\phi _1.
$$
Hence Lemma \ref{Lem:InvOpProp}
shows that $\psi _1^{-1}=T(\phi _1)$,
and by \eqref{Eq:InvOpProp2} we get
$$
\nm {\psi _1^{-1}\cdot f - \psi _0^{-1}f}
{\maclM}
=
\nm {T(\phi _1)\cdot f - f}{\maclM}
\le
(1-2^{-p})^{-\frac 2p}\, \frac {2r+1}
{4r}\nm {\phi _1\cdot f-f}{\maclM}
\le
2^{-1}\ep ,
$$
and the result follows for $n=1$.

\par

Suppose the result holds true for
$n\le m$, $m\ge 1$. We shall prove the
result for $n=m+1$. Therefore, suppose 
that there are $\phi _j$,
$j=1,\dots ,m$ such that
$\psi _j^{-1}\in \maclB _E$ exists and 
satisfies
\eqref{Eq:RefInvProp2} for $j=1,\dots ,m$.

\par

Let $\phi \in B_{r,\maclB}$ be arbitrary 
and set
$$
\psi 
=
\frac {(2r)^{m+1}}{(2r+1)^{m+1}}\varrho 
+
\left (\sum _{k=1}^m
\frac {(2r)^{k-1}}
{(2r+1)^k}\phi _k\right )
+
\frac {(2r)^m}{(2r+1)^{m+1}}\phi,
$$
which we can write
\begin{equation}\label{Eq:psiReform}
\psi
= 
\left (
\frac {2r}{2r+1}\varrho
+
\frac {1}{2r+1}\phi
\right )\chi 
=
T(\phi)^{-1}\chi,
\end{equation}
where
$$
\chi
=
\frac {(2r)^{m}}{(2r+1)^{m}}\varrho
+
\left (\sum _{k=1}^m\frac {(2r)^{k-1}}
{(2r+1)^k}T(\phi )\phi _k\right ).
$$
This gives
\begin{equation}
\chi -\psi _m
=
\sum _{k=1}^m\frac {(2r)^{k-1}}{(2r+1)^k}(T(\phi )\phi _k-\phi _k).
\end{equation}
By Lemma \ref{Lem:InvOpProp} with $\maclM$
replaced by $\maclB _E$ it follows that
\begin{align}
\nm {\chi -\psi _m}{\maclB _E}^p
&\le
\sum _{k=1}^m
\frac {(2r)^{p(k-1)}}{(2r+1)^{pk}}
\nm {T(\phi )\phi _k-\phi _k}{\maclB _E}^p
\notag
\\[1ex]
&\lesssim
\sum _{k=1}^m\frac {(2r)^{p(k-1)}}{(2r+1)^{pk}}
\nm {\phi \phi _k-\phi _k}{\maclB _E}^p
\label{Eq:EstTauPsiM}
\end{align}

\par

Now let $\ep _1>0$ be arbitrary. Then it 
follows from Definition 
\ref{Def:ApprLeftUnit}, Lemma
\ref{Lem:InvOpProp} and 
\eqref{Eq:EstTauPsiM}
that we may choose $\phi$ such that
\begin{align*}
\nm {\chi -\psi _m}{\maclB _E}
&< \ep _1,
\quad
\nm {\chi ^{-1}-\psi _m^{-1}}{\maclB _E}
< \ep _1
\intertext{and}
\nm {\phi \cdot f -f}{\maclM}
&< \ep _1.
\end{align*}
Here we have also used the fact that the invertibility
of $\psi _m$ implies that $\chi$ is 
invertible provided
$\nm {\chi -\psi _m}{\maclB _E}$ is
small enough, and that the invertible 
elements in $\maclB _E$
is an open set
$\Omega \subseteq \maclB _E$,
leading to that the map
$\chi \mapsto \chi ^{-1}$
is a homeomorphism on $\Omega$.
Here we have also used the fact that we may choose
$\phi \in B_{r,\maclB}$ such that
$$
\nm {\phi \phi _k-\phi _k}{\maclB _E}
\quad \text{and}\quad
\nm {\phi \cdot f-f}{\maclM}
$$
can be made arbitrarily small.

\par

This gives
Hence, for such choice of $\phi$
it follows that $\psi$ in
\eqref{Eq:psiReform} is invertible,
and

\begin{align*}
\nm {\psi ^{-1}\cdot f - \psi _m^{-1}\cdot f}{\maclM}^p
&=
\nm {\chi ^{-1}T(\phi )\cdot f - \psi _m^{-1}\cdot f}{\maclM}^p
\\[1ex]
&\le
\nm {\chi ^{-1}T(\phi )\cdot f
- \psi _m^{-1}T(\phi )\cdot f}{\maclM}^p
+
\nm {\psi _m^{-1}T(\phi )\cdot f
- \psi _m^{-1}\cdot f}{\maclM}^p
\\[1ex]
&\lesssim
\nm {\chi ^{-1}-\psi _m^{-1}}{\maclB _E}^p\nm {T(\phi )}{\maclB _E}^p
\nm f{\maclM}^p
+
\nm {\psi _m^{-1}}{\maclB _E}^p\nm {T(\phi )\cdot f-f}{\maclM}^p
\\[1ex]
&\lesssim
\nm {\chi ^{-1}-\psi _m^{-1}}{\maclB _E}^p\nm f{\maclM}^p
+
\nm {\phi \cdot f-f}{\maclM}^p
<
2^{-(m+1)}\ep ,
\end{align*}
provided $\ep _1$ was chosen small enough. 
Hence, by choosing
$\phi _{m+1}=\phi$ and $\psi _{m+1}=\psi$,
we have proved the result for $n=m+1$, 
when it holds for
$n\le m$. The result now follows for 
general $n$ by induction.
\end{proof}

\par

\begin{proof}
%
Our strategy is to use the approximate identity to construct
a sequence $ \{\psi_n \} _{n=1}^\infty$ of invertible elements in $\maclB_E$
such that the sequence $\{ \psi_n\} _{n=1}^\infty$ converges to an element
in $\maclB$ and $\{\psi_n^{-1} f\}$ also converges in
$\maclM$, though $\{ \psi^{-1}_n\}$ is an unbounded sequence
in $\maclB$. The elements $\psi$ and $g$ then appear
from these sequences.
%

\par

The result is obviously true for $f=0$. 
Therefore suppose that
$f\neq 0$. Let $k,n\in \mathbf N$,
$\phi _n$ and $\psi _n$ be as in Lemma
\ref{Lem:RefInvProp}, and set
$g_n=\psi _n^{-1}\cdot f$. Then
$\psi _n\cdot g_n=f$,
and \eqref{Eq:RefInvProp2} gives
\begin{equation}\label{Eq:CauchySeq}
\begin{aligned}
\nm {g_{n+k}-g_{n}}{\maclM}^p
&\le
\sum _{j=0}^{k-1}\nm {g_{n+j+1}-g_{n+j}}
{\maclM}^p
\\[1ex]
&<
\ep ^p\sum _{j=0}^{k-1} 2^{-p(n+j)}
<
\ep ^p\frac {2^{-pn}}{1-2^{-p}}\to 0,
\end{aligned}
\end{equation}
as $k,n\to \infty$. Hence $\{ g_n\} _{n=1}^\infty$ is a Cauchy
sequence which has a limit $g\in \maclM$.

\par

By \eqref{Eq:DefApprLeftUnitMod} it follows that
$f\in \maclM _f$. Hence
$g_n=\psi _n^{-1}\cdot f\in \maclM _f$ and
$g\in \maclM _f$, since $\maclB _E\cdot \maclM _f \subseteq
\maclM _f$ and that $\maclM _f$ is closed.

\par

By letting $n=0$ and $k$ tending to $\infty$,
\eqref{Eq:CauchySeq} gives
$$
\nm {g-f}{\maclM}
\le
\frac \ep{(1-2^{-p})^{\frac 1p}},
$$
and (2) follows after redefining $\ep >0$. It follows
from \eqref{Eq:RefInvProp1} that
\begin{equation}\label{Eq:LimPsi}
\psi \equiv \lim _{n\to \infty}\psi _n
=
\sum _{k=1}^\infty
\frac {(2r)^{k-1}}{(2r+1)^k}\phi _k,
\end{equation}
where the limit exists in $\maclB _E$. Since 
the series on the right-hand side in 
\eqref{Eq:LimPsi} converges in $\maclB$,
we have $\psi \in \maclB$.
Furthermore,
\begin{align*}
\nm \psi{\maclB}^p
&\le
\sum _{k=1}^\infty
\left (\frac {(2r)^{k-1}}{(2r+1)^k}
\nm {\phi _k}{\maclB}
\right )^p
\\[1ex]
&\le
\sum _{k=1}^\infty
\left (\frac {(2r)^{k-1}}{(2r+1)^k}r
\right )^p
=
\frac {r^p}{(2r+1)^p-(2r)^p}
=
r_0^p,
\end{align*}
which shows that $\psi \in B_{r_0,\maclB}$.

\par

We also have
$$
f=\lim _{n\to \infty} \psi _n\cdot g_n
=
\psi \cdot g,
$$
and the result follows.
\end{proof}

\par

\section{Proofs of some basic properties for
a broad class of modulation spaces}\label{App:B}

\par

First we prove that $M(\omega ,\mascB)$
is independent of the choice of window function
$\phi \in \Sigma _1(\rr d)$ in the
definition of modulation space norm \eqref{Eq:ModNorm}.

\par

\begin{proof}[Proof of Proposition \ref{Prop:ModNormInv}]
First we prove (1).
Let $v$ and $v_0$ be submultiplicative such that $\omega$
is $v$-moderate and $\mascB$ is a quasi-Banach space
Also let $\phi _1,\phi _2\in \Sigma _1(\rr d)$ satisfy
$\nm {\phi _1}{L^2}=\nm {\phi _2}{L^2}=1$, and
let $f\in \Sigma _1'(\rr d)$. Then
$$
|V_{\phi _2}f(x,\xi )|
\le
(2\pi )^{-\frac d2}
|V_{\phi _1}f| * |V_{\phi _2}\phi _1| .
$$
Hence, if $v\in \mascP _E(\rr {2d})$ is submultiplicative
and chosen such that $\omega$ is $v$-moderate,
$F_j=|V_{\phi _j}f |\cdot \omega $, $j=1,2$ and
$\Phi = |V_{\phi _2}\phi _1|\cdot v$, then
$$
F_1,F_2,\Phi \ge 0
\quad \text{and}\quad
F_2 \lesssim F_1*\Phi .
$$
We also replace $\omega$, $v$ and $v_0$
with smooth equivalent weights, which is possible in
view of e.{\,}g. \cite[Lemma 2.8]{Toft10}.

\par

Hence, if $Q=[0,1]^{2d}$ and suitable
$Y_j\in Q$, $j\in \zz {2d}$, then the mean-value
theorem gives 
\begin{equation}
\begin{alignedat}{1}
\nm {V_{\phi _2}f\, \omega}{\mascB}^p
&=
\nm {F_2}{\mascB}^p
\lesssim
\nm {F_1*\Phi}{\mascB}^p
\\[1ex]
&=
\NM {\sum _{j\in \zz {2d}} \int _{j+Q}F_1(\cdo -Y)\Phi (Y)\, dY}
{\mascB}^p
\\[1ex]
&\le
\sum _{j\in \zz {2d}} \NM {\int _{j+Q}F_1(\cdo -Y)\Phi (Y)\, dY}
{\mascB}^p
\\[1ex]
&\le
\sum _{j\in \zz {2d}} \left (
\NM {F_1(\cdo -j-Y_j)\int _{j+Q}\Phi (Y)\, dY}
{\mascB}
\right )^p
\\[1ex]
&\le
\sum _{j\in \zz {2d}} \left (
\nm {F_1}{\mascB} v_0(j)\int _{j+Q}\Phi (Y)\, dY
\right )^p
\asymp
\nm {F_1}{\mascB},
\end{alignedat}
\end{equation}
where the last relation follows from the fact that 
$V_{\phi _2}\phi _1\in \Sigma _1(\rr {2d})$, which
implies that $0\le \Phi (X)\lesssim e^{-r|X|}$,
for every $r>0$.

\par

Summing up we have proved
$$
\nm {V_{\phi _2}f\, \omega}{\mascB}
\lesssim
\nm {V_{\phi _1}f\, \omega}{\mascB},
$$
and by symmetry we obtain the reversed
inequality. Hence
$$
\nm {V_{\phi _2}f\, \omega}{\mascB}
\asymp
\nm {V_{\phi _1}f\, \omega}{\mascB},
$$
and (1) follows.

\par

Since it is obvious that
$M(\omega ,\mascB)$ and $M_0(\omega ,\mascB)$
are quasi normed spaces of order $p$,
the assertion (2) will follow if we prove that
$M(\omega ,\mascB)$ is complete.

\par

Let $\phi \in \Sigma _1(\rr d)\setminus 0$ be fixed
and suppose that $\{f_j\} _{j=1}^\infty$ is a Cauchy
sequence in $M(\omega ,\mascB)$. Then
$\{V_\phi f_j\} _{j=1}^\infty$ is a Cauchy
sequence in $\mascB _{(\omega )}$. Since
$\mascB _{(\omega )}$ is complete, there
is an $F\in \mascB _{(\omega )}$ such that
\begin{equation}\label{Eq:LimitCauchySeq}
\lim _{j\to \infty} \nm {V_\phi f_j -F}{\mascB _{(\omega )}}=0,
\end{equation}
Since $M(\omega ,\mascB )$ is normal we have
$M(\omega ,\mascB )\hookrightarrow M^\infty _{(1/v)}(\rr d)$
for some $v\in \mascP _E(\rr {2d})$.
By the first convergence in combination with the fact that
$M^\infty _{(1/v)}(\rr d)$ is a Banach space and thereby
complete, we have
$$
\lim _{j\to \infty} \nm {V_\phi f_j-V_\phi f}{L^\infty _{(1/v)}}=0,
$$
for some $f\in M^\infty _{(1/v)}(\rr d)\subseteq \Sigma _1'(\rr d)$,
with convergence in $L^\infty _{(\omega /v_0)}(\rr {2d})$.

\par

From these limits it follows that $F=V_\phi f$ a.{\,}e., giving that
$$
\lim _{j\to \infty}
\nm {f_j-f}{M(\omega ,\mascB)}
=
\lim _{j\to \infty}
\nm {V_\phi f_j-V_\phi f}
{\mascB _{(\omega )}}
=
\lim _{j\to \infty}
\nm {V_\phi f_j-F}{\mascB _{(\omega )}}
=0.
$$
A combination of the latter limit and the first limit in
\eqref{Eq:LimitCauchySeq} now implies that
$f\in M(\omega ,\mascB)$ and
$$
\lim _{j\to \infty} \nm {f_j-f}{M(\omega ,\mascB)}=0.
$$
Consequently, the desired completeness is obtained
and the result follows.
\end{proof}

\par

\begin{proof}[Proof of Proposition
\ref{Prop:ModEmb}]
The assertion (1) is obviously true, since
${\omega_2}/{\omega_1}\in
L^\infty (\rr {2d})$ implies that
$$
\nm f{M^\sharp (\omega _2,\mascB )}
\lesssim
\nm f{M^\sharp (\omega _2,\mascB )}
\quad \text{and}\quad
\nm f{M(\omega _2,\mascB )}
\lesssim
\nm f{M(\omega _2,\mascB )}.
$$
We need to prove (2).

\par

Therefore assume that $v$ is bounded.
If ${\omega_2}/{\omega_1}\in
L^\infty (\rr {2d})$, then (1) shows that
the mappings in
\eqref{Eq:ModEmb} are continuous injections.

\par

Suppose instead ${\omega_2}/{\omega_1}
\notin L^\infty (\rr {2d})$. Then there are
$\{ (x_k,\xi _k)\} _{k=1}^\infty \subseteq
\rr {2d}$ such that
$$
\frac {\omega _2(x_k,\xi _k)}
{\omega _1(x_k,\xi _k)}\ge k,
\qquad k\in \mathbf Z_+.
$$
Let
$$
f_k(x) = \omega _2(x_k,\xi _k)^{-1}
e^{-\frac 12 |x-x_k|^2}e^{i\scal x{\xi _k}},
\quad
x\in \rr d,
$$
and choose
$\phi (x)= e^{-\frac 12 |x|^2}
\in \Sigma _1(\rr d)$ as window function. Since
$f_k\in \Sigma _1(\rr d)$ in view of Proposition
\ref{Prop:GSSpacesChar},
it follows that $f_k\in M(\omega _j,\mascB )$ in view of
\eqref{Eq:GSinMod}.

\par

We have
$$
|V_\phi f_k(x,\xi )|
=
C\omega _2(x_k,\xi _k)^{-1}F(x-x_k,\xi -\xi _k),
\qquad
F(x,\xi )=
e^{-\frac 14(|x|^2+|\xi |^2)},
$$
for some constant $C>0$ which is independent
of $k$. This gives
\begin{equation}\label{Eq:ModNormGaussFunc}
\begin{aligned}
\nm {f_k}{M(\omega _j,\mascB )}
&\asymp
\nm {V_\phi f_k\cdot \omega _j}{\mascB}
\asymp
\omega _2(x_k,\xi _k)^{-1}
\nm {F(\cdo -(x_k,\xi _k) )\cdot \omega _j}{\mascB}
\\[1ex]
&\asymp
\omega _2(x_k,\xi _k)^{-1}
\nm {F\cdot \omega _j(\cdo +(x_k,\xi _k) )}{\mascB},
\qquad j=1,2.
\end{aligned}
\end{equation}
The last relation follows from the fact that $v$
is bounded.

\par

If $v_0\in \mascP _E(\rr {2d})$ is chosen
such that $\omega _1$ and $\omega _2$ are
$v_0$ moderate, then
\eqref{Eq:ModNormGaussFunc} gives
\begin{multline*}
\nm {f_k}{M (\omega _1,\mascB )}
\asymp
\omega _2(x_k,\xi _k)^{-1}
\nm {F\cdot \omega _1(\cdo +(x_k,\xi _k) )}{\mascB}
\\[1ex]
\gtrsim
\frac {\omega _1(x_k,\xi _k)}{\omega _2(x_k,\xi _k)}
\nm {F/v_0}{\mascB}
\asymp
\frac {\omega _1(x_k,\xi _k)}{\omega _2(x_k,\xi _k)}
\ge k \to \infty
\end{multline*}
as $k\to \infty$. In similar way we get
$$
\nm {f_k}{M (\omega _2,\mascB )}
\asymp
\omega _2(x_k,\xi _k)^{-1}
\nm {F\cdot \omega _1(\cdo +(x_k,\xi _k) )}{\mascB}
\\[1ex]
\lesssim 
\frac {\omega _2(x_k,\xi _k)}{\omega _2(x_k,\xi _k)}
\nm {F\cdot v_0}{\mascB} \le C,
$$
for some constant $C>0$. This shows that
$\{ f_k \} _{k=1}^\infty$ is a bounded set
in $M(\omega _2,\mascB )$ but an
unbounded set in $M(\omega _1,\mascB )$.
Hence the mappings in
\eqref{Eq:ModEmb} are discontinuous, and the
result follows.
\end{proof}

\par

\section{Proof of certain continuity for
convolution properties in Wiener amalgam spaces}
\label{App:C}

\par

In this appendix we present a proof of
Proposition \ref{Prop:ConvWien}.

\par

\begin{proof}[Proof of Proposition
\ref{Prop:ConvWien}]
Since the convolution is commutative on
$\Sigma _1(\rr d)$, it suffices to prove
the assertions for
$f\in \WL ^{p,q}_{(\omega )}(\rr d)$
and
$g\in \WL ^{1,r}_{(v)}(\rr d)$.
We only prove the result in the case when 
$p,q<\infty$. The assertion
(1) in the cases when $p=\infty$ or 
$q=\infty$ follows by similar arguments
and is left for the reader.

\par

We have that $\Sigma _1(\rr d)$ is dense in
$\WL ^{1,r}_{(v)}(\rr d)$
and that $h=f*g$ and its derivatives belong to
$\Sigma _1'(\rr d)\cap C^\infty (\rr d)$ when
$f\in \WL ^{p,q}_{(\omega )}(\rr d)$
and and $g\in \Sigma _1(\rr d)$.
Hence, the assertion (1) follows if we prove that 
\eqref{Eq:WienProdEst} holds for
$f\in \WL ^{p,q}_{(\omega )}(\rr d)$
and
$g\in \Sigma _1(\rr d)$.

\par

Therefore suppose that
$f\in \WL ^{p,q}_{(\omega )}(\rr d)$
and
$g\in \Sigma _1(\rr d)$, and let $Q=[0,1]^d$,
$Q_M=[-1,1]^d$,
$$
c_f(k)=\nm f{L^p(k+Q_M)}\omega (k),
\quad
c_g(k)=\nm g{L^1(k+Q)}v (k)
\quad \text{and}\quad
c_h(k)=\nm h{L^p(k+Q)}\omega (k).
$$
Then
$$
\nm f{\WL ^{p,q}_{(\omega )}}
\le
\nm {c_f}{\ell ^q(\zz d)}
\le 2^dc_v\nm f{\WL ^{p,q}_{(\omega )}},
\quad
\nm g{\WL ^{1,r}_{(v)}}
=
\nm {c_g}{\ell ^r(\zz d)}
$$
and
$$
\nm h{\WL ^{p,q}_{(\omega )}}
=
\nm {c_h}{\ell ^q(\zz d)}.
$$
A combination of these relations and Minkowski's 
inequality gives
\begin{align*}
c_h(k)
&=
\left (
\int _{k+Q}|f*g(x)|^p\, dx 
\right )^{\frac 1p}\omega (k)
\\[1ex]
&\le
\left (
\int _{k+Q}\left |
\sum _{j\in \zz d} \int _{j+Q}|f(x-y)||g(y)|\, dy
\right |^p\, dx
\right )^{\frac 1p}\omega (k)
\\[1ex]
&\le
\sum _{j\in \zz d} \int _{j+Q}\left (
\int _{k+Q}\left |
f(x-y)
\right |^p\, dx
\right )^{\frac 1p}
|g(y)|\, dy\, 
\omega (k)
\\[1ex]
&=
\sum _{j\in \zz d} \int _{j+Q}\left (
\int _{k-y+Q}\left |
f(x)
\right |^p\, dx
\right )^{\frac 1p}
|g(y)|\, dy\, 
\omega (k)
\\[1ex]
&\le
\sum _{j\in \zz d} \int _{j+Q}\left (
\int _{k-j+Q_M}\left |
f(x)
\right |^p\, dx
\right )^{\frac 1p}
|g(y)|\, dy\, 
\omega (k)
\\[1ex]
&\le
\sum _{j\in \zz d} \left (\left (
\int _{k-j+Q_M}\left |
f(x)
\right |^p\, dx
\right )^{\frac 1p} \omega (k-j)\right )
\left (\int _{j+Q} |g(y)|\, dy\, 
v(j)\right )
\\[1ex]
&=
(c_f*c_g)(k).
\end{align*}
Hence, by Young's inequality we obtain
\begin{align*}
\nm h{\WL ^{p,q}_{(\omega )}}
&=
\nm {c_h}{\ell ^q} \lesssim \nm {c_f*c_g}{\ell ^q}
\le
\nm {c_f}{\ell ^q}\nm {c_g}{\ell ^{\min (q,1)}}
\\[1ex]
&\le
\nm {c_f}{\ell ^q}\nm {c_g}{\ell ^r}
\le
2^dc_v\nm f{\WL ^{p,q}_{(\omega )}}
\nm g{\WL ^{1,r}_{(v)}}
\end{align*}
This gives \eqref{Eq:WienProdEst} when
$f\in \WL ^{p,q}_{(\omega )}(\rr d)$
and $g\in \Sigma _1(\rr d)$, and (1) follows.
\end{proof}

\par

\end{document}